\newif\ifarxiv
\arxivtrue 

\documentclass[12pt]{article}
\usepackage{graphicx, epsfig, psfrag}
\usepackage{amsfonts,amsmath,amssymb,amsbsy,amsthm,xfrac,cite}
\usepackage{bm}
\usepackage{color}
\usepackage{float}
\definecolor{darkblue}{rgb}{0,0,0.5}
\usepackage[colorlinks=true, linkcolor=darkblue, urlcolor=darkblue, citecolor=darkblue]{hyperref}
\usepackage{mathrsfs}
\usepackage{subfigure}
\ifarxiv
	\usepackage[left=0.75in, top=0.75in, right=0.5in, bottom=0.5in, includefoot]{geometry}
\else
	\usepackage[paperwidth=7.4in, paperheight=9.565in, left=0.5in, top=0.75in, right=0.5in, bottom=0.5in, includefoot]{geometry}
\fi
\usepackage[normalem]{ulem}
\usepackage{multirow}

\numberwithin{equation}{section}

\usepackage{ifpdf}
\ifpdf
\else
\renewcommand{\includegraphics}[1]{\framebox{Graphics Placeholder}}
\renewcommand{\includegraphics}[2][1]{\framebox{Graphics Placeholder}}
\fi

\usepackage{environments}

\newcommand{\smfrac}[2]{{\textstyle \frac{#1}{#2}}}
\newcommand{\half}{\mathstrut^{1\!\!}/_{\!2}}
\newcommand{\threehalfs}{\mathstrut^{3\!\!}/_{\!2}}

\newcommand{\oprod}{\sideset{}{^\circ}\prod}
\newcommand{\smalloprod}{\sideset{}{^\circ}{\textstyle \prod}}

\renewcommand{\paragraph}[1]{\subsubsection{#1}}

\newcommand{\<}{\langle}
\renewcommand{\>}{\rangle}
\newcommand{\weakto}{\rightharpoonup}

\newcommand{\del}{\delta}

\newcommand{\ddel}{{\delta^2\hspace{-1pt}}}
\newcommand{\dddel}{{\delta^3\hspace{-1pt}}}

\definecolor{mlcol}{rgb}{0, 0.4, 0}
\definecolor{ascol}{rgb}{0.7, 0, 0}

\newcommand{\eps}{\epsilon}
\newcommand{\C}{{\rm C}}

\renewcommand{\i}{{\rm i}}

\newcommand{\bc}{{\rm bc}}

\newcommand{\lin}{{\rm lin}}

\newcommand{\err}{{\rm err}}

\newcommand{\sgn}{{\rm sgn}}

\newcommand{\argmin}{\mathop{\mathrm{argmin}}}

\newcommand{\dd}{{\rm d}}
\newcommand{\du}{{\rm d}u}
\newcommand{\dv}{{\rm d}v}
\newcommand{\dx}{{\rm d}x}
\newcommand{\dy}{{\rm d}y}

\newcommand{\bbC}{{\mathbb C}}

\newcommand{\bbN}{{\mathbb N}}

\newcommand{\bbR}{{\mathbb R}}
\newcommand{\bbZ}{{\mathbb Z}}

\newcommand{\calA}{{\mathcal A}}
\newcommand{\calB}{{\mathcal B}}
\newcommand{\calD}{{\mathcal D}}

\newcommand{\calL}{{\mathcal L}}

\newcommand{\calP}{{\mathcal P}}
\newcommand{\calU}{{\mathcal U}}

\newcommand{\tildeu}{{\tilde{u}}}

\newcommand{\per}{{\rm per}}
\newcommand{\free}{{\rm free}}

\newcommand{\mF}{{\sf F}}

\usepackage{accents}
\newlength{\dhatheight}
\newlength{\dtildeheight}

\newcommand{\ttilde}[1]{%
    \settoheight{\dtildeheight}{\ensuremath{\tilde{#1}}}%
    \addtolength{\dtildeheight}{-0.15ex}%
    \tilde{\vphantom{\rule{1pt}{\dtildeheight}}%
    \smash{\tilde{#1}}}}

\begin{document}

\title{Accuracy of computation of crystalline defects at finite temperature}

\author{Alexander V. Shapeev\thanks{University of Minnesota. AS was supported in part by the
DOE Award DE-SC0002085. (ashapeev@umn.edu)}~
and
Mitchell Luskin\thanks{University of Minnesota. ML was supported in part by the NSF PIRE Grant OISE-0967140,
NSF Grant 1310835, DOE Award DE-SC0002085, AFOSR Award
FA9550-12-1-0187, and ARO MURI Award W911NF-14-1-0247.
(luskin@umn.edu)}}


\date{\today}




\maketitle

\begin{abstract}
The present paper aims at developing a theory of computation of crystalline defects at finite temperature.
In a one-dimensional setting we introduce Gibbs distributions corresponding to such defects and rigorously establish their asymptotic expansion.
We then give an example of using such asymptotic expansion to compare the accuracy of computations using the free boundary conditions and using an atomistic-to-continuum coupling method.

\end{abstract}

\section{Introduction}

Many critical materials phenomena are determined by materials defects.
For instance, plasticity in crystals is determined by motions of dislocations, while the response to radiation damage is determined by the motion of vacancies and interstitials.
Therefore predicting such properties computationally requires computing defects.

There has been considerable progress recently in developing a numerical analysis theory of defect computation at zero temperature 
(we refer to recent a review \cite{LuskinOrtner2013acta} and
 more recent papers \cite{bqce12,EhrlacherOrtnerShapeev2013preprint,LiLuskinOrtnerEtAl2014,OrtnerShapeev2013}), as well as new numerical 
methods whose accuracy could be rigorously quantified \cite{bqce11,MakridakisMitsoudisRosakis2012,OlsonBochevLuskinEtAl,OrtnerZhang2013consistent,OrtnerZhang2012consistent,Shapeev2011,Shapeev2012}.
Efforts towards rigorous analysis of the accuracy of computation of defects at finite temperature, however, are just beginning \cite{LegollOrtnerZhang}.

The present work aims at developing a theory of finite-temperature defect computation.
In particular, we were motivated by the so-called hot-QC method \cite{DupuyTadmorMillerEtAl2005, TadmorLegollKimEtAl2013,KimLuskinPerezEtAl2014hyper}.
This method systematically reduces the degrees of freedom of a large system to a small system, aiming at resolving atoms only near a defect, while capturing some effective interaction of these atoms with the distant atoms (as opposed to just truncating a domain and performing a full atomistic simulation).
The tool for deriving such a method is the expansion of Gaussian-like integrals in terms of powers of small temperature, and is well-justified for systems of fixed size.
However the accuracy of such methods in the thermodynamic limit (when the system size grows to infinity) has not been rigorously studied.

In this paper, under assumptions of one dimension, relatively short-range interaction (essentially, next nearest-neighbor interaction), and the free boundary conditions, we give a definition of Gibbs measures for defects in an infinite lattice.
We then study the accuracy of approximation of such infinite-dimensional Gibbs measures by their finite-dimensional counterparts associated to different computational methods.
In essence, we develop a rigorous asymptotic expansion of such Gibbs measures and establish the convergence of the expansion terms.
Since the expansion terms have a simpler structure than the original Gibbs measure, it becomes possible to quantify the error for small, but finite, temperatures.
This may be especially useful in two or three spatial dimensions: although this work does not give a rigorous justification for this case, we propose that one could simply \emph{conjecture} such expansions are valid in two or three dimensions, and hence study the convergence of the corresponding expansion terms.

The structure of the paper is as follows.
In Section \ref{sec:defect}, we introduce some definitions, and in Section \ref{sec:defect-at-nonzero-temperature} we formulate the main result, 
which we prove in Sections \ref{sec:finite_lattice} and \ref{sec:infinite}.
In Section \ref{sec:applications}, we give an example of how to apply the developed theory for two computational methods.
We summarize and discuss the results of the paper in Section \ref{sec:discussion-and-conclusion}.

\section{Main Result}\label{sec:main}

\subsection{Atomistic Model and a Defect}\label{sec:defect}

Instead of working with atomistic displacements $\tilde{u}$, we will work with a discrete strain $u:\bbZ\to\bbR$, $u(\xi) = \tildeu(\xi+\half)-\tildeu(\xi-\half)$. 
We let $\calB_N := \{-N, \ldots, N\}$ ($N\in\bbZ_+$) and $\calB_\infty := \bbZ$ index nearest-neighbor bonds.
Let
\[
\calU_{N} := \bbR^{\calB_N}
\qquad \forall N\in\bbZ_+\cup\{\infty\}
\]
be the space of strains.
The lattice (of non-boundary atoms) is then
\[
\calL_N := \{ -N+\smfrac12, \ldots, N-\smfrac12 \}
\qquad\text{and}\qquad
\calL_\infty := \bbZ + \smfrac12
.
\]

We assume the energy of the system has the form
\begin{align*}
	E_{\theta,N}(u)
:=~&
	\calP(u) +
	\sum_{\xi\in\calL_N} V(u_{\xi-\half}, u_{\xi+\half} )
	+ V_{\theta}^{\bc}(u_N)
	+ V_{\theta}^{\bc}(u_{-N})
\qquad \forall N<\infty,
\\
	E_\infty(u)
:=~&
	\calP(u) +
	\sum_{\xi\in\calL_\infty} V(u_{\xi-\half}, u_{\xi+\half} )
.
\end{align*}
Here $\calP(u)$ models a defect and is assumed to be localized within $\calB_{K_0}$, i.e., $\calP(u) = \calP(u|_{\calU_{K_0}})$, where $K_0$ is the size of the defect core.
The interaction potential $V(x,y)$, $(x,y)\in\bbR^2$, models the atomistic interactions and
\[
V^\bc_\theta(x) = V^{\bc,(0)}(x)
+ \theta V^{\bc,(1)}(x) +...
+ \theta^{M-1} V^{\bc,(M-1)}(x)
+ \theta^{M} V^{\bc,(M^+)}_\theta(x)
,\quad x\in \bbR
\]
models the boundary conditions for some $M\in\bbN$, where the last term $V^{\bc,(M^+)}_\theta(x)$ gathers the $M$-th order term and all higher-order terms.

One choice of $\calP$ would be restricting $\calP=0$ near a given local minimum of the energy and $\calP=+\infty$ far from it, thus effectively restricting the Gibbs distribution to a neighborhood of the defect---see below for a discussion of the assumption \eqref{eq:defect-global}.
Another example of $\calP$ would be an interstitial, e.g., at some site $\xi_0$ whose site energy is $V^{\rm i}(x,y)$.
Then one chooses $\calP(u) = V^{\rm i}(u_{\xi_0-\half},u_{\xi_0+\half}) - V(u_{\xi_0-\half},u_{\xi_0+\half})$.
For more discussion on how to model various defects with a such a potential $\calP$, we refer to \cite{EhrlacherOrtnerShapeev2013preprint, LiOrtnerShapeevEtAl2014}.

The two examples of the boundary condition we will consider are:
\begin{enumerate}
\item Free BCs: $V^{\bc}_\theta \equiv 0$
\item A/C coupling-type B.C.: $V^{\bc}_\theta$ comes from eliminating the continuum DoFs as will be introduced in Section \ref{sec:applications:QC}.
Such methods were proposed in \cite{DupuyTadmorMillerEtAl2005, TadmorLegollKimEtAl2013}.
\end{enumerate}
In particular, the latter boundary condition would depend on $\theta$---hence we allow such dependence in $V^\bc_\theta(y)$.

\subsubsection*{Assumptions}

We assume the existence of the subject of our investigation, a defect, described by the corresponding strain $u^*_\infty\in\ell^2(\bbZ)$.
\emph{We will identify $u^* = u^*_\infty$ where it will cause no confusion.
In this case the $n$-th component of $u^*$ will be denoted as $u^*_{,n}$.
}%
Note that $u^*\in\ell^2(\bbZ)$ implies $u^*_{,n}\to 0$ as $n\to\infty$.
Also, we assume that the defect is point-symmetric, $u^*_{,-n} = u^*_{,n}$, for the sole reason of reducing technicalities in the proofs.

Next, we make some regularity assumptions on the interaction potentials, $\calP$, $V$ and $V^\bc_\theta$, near the defect and assumptions on the growth of potentials far from the defect (to effectively restrict a system to a neighborhood of a defect).

In particular, we assume that the potentials can be $+\infty$, but are finite and have a certain number (depending on $M\in\bbN$) of derivatives near the defect $u^*$:
\begin{align}
\label{eq:orig_ass:V_meas}
\calP, V, V_\theta^{\bc,(m)}, V_\theta^{\bc,(M^+)} \text{ are}& \text{ measurable with values in $\overline{\bbR}$},
\\ \label{eq:orig_ass_X}
\{(x,y) : V(x,y)<+\infty \}=~& \calA\times \calA
\text{ for some $\calA\subset \bbR$}
,
\\
\label{eq:orig_ass:P_bound}
\calP(u)
\geq~&
-C_\calP (1+\|u\|),
\qquad\forall u\in\calU_{K_0},
\\
\label{eq:orig_ass:P_bound_loc}
\|\calP\|_{\C^{2M+3}(B_{r_V})}
\leq~&
c_V,
\\
\label{eq:orig_ass:V_bound}
\|V\|_{\C^{2M+3}(B_{r_V})}
\leq~&
c_V,
\\ \label{eq:orig_ass:Vbc_bound}
\|V^{\bc,(m)}\|_{\C^{2M-2 m+3}(B_{r_V})}
\leq
c_V,
\qquad&
\|V_{\theta}^{\bc,(M^+)}\|_{C(B_{r_V})}
\leq
c_V,
\end{align}
for all $m=0,\ldots,M-1$ and $\theta>0$,
for some $c_V>0$ and $r_V>\|u^*\|_{\ell^{\infty}}$.
Here $B_r$ denotes a ball with radius $r$ and center $0$ in some metric space which will be clear from the context, $\overline{\bbR} = \bbR\cup\{\pm\infty\}$, and measurably is understood in the Lebesgue's sense.
Allowing $V$ or $\calP$ to be $+\infty$ away from the defect effectively leads to restricted Gibbs distributions \cite{KimLuskinPerezEtAl2014hyper,TadmorLegollKimEtAl2013}.
Note that \eqref{eq:orig_ass_X} is a technical assumption that means, essentially, that an admissible region for $V$ is $\calA\times\calA$, where $\calA$ is some set of admissible strains (e.g., one may choose $\calA=\{x:x>0\}$);

Next, we fix the energy of $u=0$ to be zero:
\begin{align} \label{eq:orig_ass:Vzero}
\calP(0) =  V(0,0) = V^\bc_\theta(0)
=~&
0.
\end{align}
Together with the regularity assumptions this yields that $E_\infty \in \C^{2M+3}(B_{r_V})$ (see \cite{EhrlacherOrtnerShapeev2013preprint, LiOrtnerShapeevEtAl2014, OrtnerTheil2013,LiOrtnerShapeevEtAl2014} for the arguments establishing such regularity of $E_\infty$, and also the proof of Lemma \ref{lem:zero_temp_approx} for the uniform regularity of $E_{0,N}$).

The following assumptions state quadratic growth of the potentials away from $u=0$,
\begin{align} \label{eq:orig_ass:V_growth}
V(y,x) = V(x, y)
\geq~&
\smfrac{\gamma_V}{2} (x^2+y^2),
\qquad\forall (x,y)\in\bbR^2,
\\ \label{eq:orig_ass:Vbc_growth1}
V^{\bc,{(0)}}\equiv 0 \text{~~or~~} V^{\bc,{(0)}}(x) \geq~& \smfrac{\gamma_\bc}{2} x^2,
\qquad\forall x\in\bbR,
\\ \label{eq:orig_ass:Vbc_growth2}
V^{\bc,(m)}(x) \geq -c_V,
\quad
V_{\theta}^{\bc,(M^+)}(x) \geq~& -c_V,
\qquad\forall x\in\bbR,
\end{align}
for some positive $\gamma_V$, $\gamma_\bc$, and $c_V$,
and therefore these assumptions more restrictive.
Indeed realistic interaction potentials do not satisfy them, however the quadratic growth assumptions are, on the one hand, a standard way of modeling the absence of reconfiguration of atoms \cite{Funaki2005stochastic}, and on the other hand are satisfied for the restricted Gibbs distribution approach.
Also, note that the first part of \eqref{eq:orig_ass:V_growth} states the point symmetry of the interaction potential, which is dictated by the underlying physical symmetry.

The next assumption on $V^\bc$ states that $u=0$ is compatible with the boundary condition at zero temperature:
\begin{align} \label{eq:orig_ass:Vbc_force}
V_y(0,0)+V^{\bc,(0)}_x(0) =~& 0.
\end{align}
Both the free and the a/c coupling-type boundary conditions we consider satisfy this requirement.
Moreover,
we claim that this assumption can in principle be removed at a cost of extra technicalities for dealing with the resulting finite-magnitude boundary layer (however, even with this assumption there are $O(\theta)$ boundary layers).

The last assumption,
\begin{align}
\label{eq:defect-global}
E_{\infty}(u^*_\infty + v)
-
E_{\infty}(u^*_\infty)
\geq~&
\smfrac{\gamma_E}{2} \|v\|_{\ell^2}^2
\qquad \forall v\in\ell^2(\bbZ)
,
\end{align}
for some $\gamma_E>0$, requires some extra discussion.
We note that this is, essentially, the only assumption on the actual defect---the previous assumptions were on the atomistic material itself.
This assumption enforces that the defect is {\it the only} minimum of $E_{\infty}$, and is formulated similarly to the assumptions \eqref{eq:orig_ass:V_growth}--\eqref{eq:orig_ass:Vbc_growth2}.
We note that if one assumes $u^*$ to only be a strongly stable equilibrium, i.e.,
\begin{align*}
\<\del E_{\infty}(u^*_\infty), v\>
=
0,
\quad
\<\ddel E_{\infty}(u^*_\infty) v, v\>
\geq
\gamma_E \|v\|_{\ell^2}^2,
\qquad \forall v\in\ell^2(\bbZ),
\end{align*}
where $\del F(u)$ denotes the first variation, or the functional derivative of $F$ at the point $u$, and hence $\<\del F(u), v\>$ denotes the directional derivative along $v$,
then \eqref{eq:defect-global} is true in a neighborhood of the defect, and hence one can define
\[
\tilde{\calP}(u) := \begin{cases}
\calP(u), & \|u-u^*\|\leq \rho, \\
+\infty, & \|u-u^*\|> \rho,
\end{cases}
\]
for some $\rho>0$, effectively restricting the Gibbs measure to such a neighborhood.
We also note that \eqref{eq:orig_ass:V_growth}--\eqref{eq:orig_ass:Vbc_growth2} automatically implies \eqref{eq:defect-global} with $\calP\equiv 0$, which corresponds to a lattice without defects.

%

Finally, we note that \eqref{eq:orig_ass:V_growth} implies $V_{xx}(0,0)=V_{yy}(0,0)>0$.
We hence fix the spatial scale by letting $V_{xx}(0,0)=V_{yy}(0,0)=1$.
It is then easy to derive from \eqref{eq:orig_ass:V_growth} that
\begin{equation}\label{eq:alpha_def}
\alpha := V_{xy}(0,0)
\end{equation}
satisfies $|\alpha|<1$.
We also introduce $\kappa=\sqrt{1-\alpha^2}$.

\subsection{Nonzero Temperature}\label{sec:defect-at-nonzero-temperature}

In order to study how well a finite lattice of atoms approximates the infinite lattice under a small but nonzero temperature, we introduce the corresponding Gibbs measures, or more generally, distributions.

\subsubsection*{Finite Lattice}
A defect at a nonzero temperature is described by its Gibbs distributions $\mu_{\theta,N}$ and $\mu_{\theta,\infty}$.
For a finite $N\in\bbN$, we define the Gibbs measure $\mu_{\theta,N}$ by its action on an observable $A\in \C^{2M}(\calU_N)$ as
\[
\<\mu_{\theta,N}, A\>_{\calU_N}
:=
\frac{%
	\int_{\calU_N} A(v) e^{-\theta^{-1} E_{\theta, N}(v)} \dv
}{%
	\int_{\calU_N} e^{-\theta^{-1} E_{\theta, N}(v)} \dv
}
=
\frac{%
	\int_{\calU_N} A(\sqrt{\theta} v) e^{-\theta^{-1} E_{\theta,N}(\sqrt{\theta} v)} \dv
}{%
	\int_{\calU_N} e^{-\theta^{-1} E_{\theta,N}(\sqrt{\theta} v)} \dv
}
.
\]
Thus, $\mu_{\theta,N}$ belongs to the space of distributions $\C^{-2M}_0(\calU_N) = (\C^{2M}(\calU_N))^*$, where $\C^m(X)$ is the space of functions with $m$ bounded uniformly continuous derivatives on $X$.

\begin{remark}
\begin{itemize}
\item[(1)]
We use the notation for distributions $\<\mu_{\theta,N}, A\>_{\calU_N}$ instead of that for measures $\int_{\calU_N} A \,\dd\mu_{\theta,N}$, since some related objects will not be measures.

\item[(2)]
We chose $\C^{2M}(\calU_N)$ for the test space instead of, e.g., $\C^{2M}_0(\calU_N)$, because in order to study a limit $N\to\infty$ it will be important that $\C^{2M}(\bbR^n)\subset \C^{2M}(\bbR^m)$ for $n<m$.
\end{itemize}
\end{remark}

We introduce the following norms and seminorms for $\mu\in \C^{-L}_0(\calU_N)$:
\[
|\mu|_{\C^{-L}_0(\calU_K)} := \sup_{A\in\\C^{L}(\calU_K)} \<\mu, A\>_{\calU_N}
\qquad\text{and}\qquad
\|\mu\|_{\C^{-L}_0} := |\mu|_{\C^{-L}_0(\calU_N)}.
\]

\subsubsection*{Infinite Lattice}

For an infinite lattice, we introduce the space of test functions
\[
\calD
:=
\bigcup_{n\in\bbN} \C^{2M}(\calU_{n}),
\]
where we implicitly assume the natural embedding $\calU_{n}\subset\calU_m$ for $n<m\leq\infty$.
We say that $f_n \to f$ in $\calD$ if $f$ and all $f_n$ belong to some $\calU_K$ and $\|f_n-f\|_{\C^{2M}(\calU_K)} \to 0$.

We then define $\mu_{\theta,\infty} \in \calD^*$ as the weak limit of $\mu_{\theta,N}$:
\[
\<\mu_{\theta,\infty}, A\>_{\calU_\infty}
:=
\lim_{N\to\infty} \<\mu_{\theta,N}, A\>_{\calU_N}
\qquad \forall A\in\calD
.
\]
We will omit the subscript $\calU_N$ in the duality pairing notation $\<\bullet, \bullet\>$.

Finally, note that existence of the limit of Gibbs measures, $\mu_{\theta,\infty}$, in the space of distributions does not immediately imply that $\mu_{\theta,\infty}$ itself is a measure.

\subsubsection*{Main Result}

The main result of this work is the following theorem.
\begin{theorem}\label{th:uniform_expansion}
\begin{itemize}
\item[(a)]
For each $N\in\bbN$ there exist distributions $\mu_{N}^{(m)}$ ($m=0,1,\ldots,M-1$) and $\mu_{\theta,N}^{(M^+)}$ such that
\begin{align} \label{eq:expansion_main}
\mu_{\theta,N}
=~&
\sum_{m=0}^{M-1}
	\theta^{m} \mu_{N}^{(m)}
+
	\theta^{M} \mu_{\theta,N}^{(M^+)}
\qquad \forall\theta>0
.
\end{align}
\item[(b)]
There exist unique distributions $\mu_{\theta,\infty}$, $\mu_{\infty}^{(m)}$, and $\mu_{\theta,\infty}^{(M^+)}$, independent of $V^\bc$ such that for any $K$ and $A\in \C^{2M}(\calU_K)$,
\[
\<\mu_{\theta,N}, A\>
\to
\<\mu_{\theta,\infty}, A\>
,\quad
\<\mu_{N}^{(m)}, A\>
\to
\<\mu_{\infty}^{(m)}, A\>
,\quad\text{and}\quad
\<\mu_{\theta,N}^{(M^+)}, A\>
\to
\<\mu_{\theta,\infty}^{(M^+)}, A\>
\]
as $N\to\infty$ for all $\theta>0$.
Moreover,
\begin{align} \notag
\mu_{\theta,\infty}
=~&
\sum_{m=0}^{M-1}
	\theta^{m} \mu_{\infty}^{(m)}
+
	\theta^{M} \mu_{\theta,\infty}^{(M^+)}
.
\end{align}
\item[(c)]
There exist $\theta_0>0$ and $C=C(K)$ both independent of $N$ such that
\begin{align}\label{eq:uniform_expansion:bounds}
\max_{0\leq m<M \rfloor} \big|\mu_{N}^{(m)}\big|_{\C^{-L}_0(\calU_K)}\leq C(K)
\quad\text{and}\quad
\big|\mu_{\theta,N}^{(M^+)}\big|_{\C^{-L}_0(\calU_K)}\leq C(K)
\end{align}
for all $\theta\in(0,\theta_0]$ and $N\in\bbN\cup\{\infty\}$.

\item[(d)] The terms $\mu_{N}^{(m)}$ are linear combinations of the Dirac delta function and its derivatives at $u^*_N$ which is the global minimum of $E_{0,N}$. Furthermore,
\begin{align} \label{eq:mu0_explicit}
\<\mu_{N}^{(0)}, A\> =~& A(u_N^*),
\\ \label{eq:mu1_explicit}
\<\mu_N^{(1)}, A\> =~&
\smfrac12 \ddel A(u_N^*) \!:\! H^{-1}_N
-
\del A(u_N^*) \cdot H_N^{-1} \big(
	\del E_N^{(1)}(u_N^*)+\smfrac12\dddel E_{0,N}(u_N^*) \!:\! H_N^{-1}
\big)
\end{align}
for all $N\in\bbN\cup\{\infty\}$,
where we denote $E_N^{(1)}(u) := V^{\bc,(1)}(u_N) + V^{\bc,(1)}(u_{-N})$ and $H_N := \ddel E_{0,N}(u_N^*)$.
\end{itemize}
\end{theorem}

\section{Applications}\label{sec:applications}

Theorem \ref{th:uniform_expansion} can be used in the following way.
Consider Gibbs measures corresponding to two different systems, described by $E_{\infty}(u)$ and $E_{\theta, N}(u)$.
We think of $E_{\infty}(u)$ as the exact energy of the infinite system and
$E_{\theta, N}(u)$ is the energy of the system which approximates the original infinite system.

Theorem \ref{th:uniform_expansion} gives us the following estimate:
\begin{align*}
|\<\mu_{\theta,N} - \mu_{\theta,\infty}, A\>|
=~&
	|A(u^*_N) - A(u^*_\infty)|
	+ \theta \Big|
		\smfrac12 \ddel A(u^*_N) \,:\, H_N^{-1}
		-
		\smfrac12 \ddel A(u^*_\infty) \,:\, H_\infty^{-1}
\\~&-
	\del A(u^*_N) \cdot \big(
	H_N^{-1} \big(\del E^{(1)}(u^*_N) + \smfrac12 H_N^{-1} \,:\, \dddel E^{(0)}(u^*_N)\big)
\big)
\\~&+
	\del A(u^*_\infty) \cdot \big(
	H_\infty^{-1} \big(\del E^{(1)}(u^*_\infty) + \smfrac12 H_\infty^{-1} \,:\, \dddel E^{(0)}(u^*_\infty)\big)
\big)
	\Big|
\\~&+
	O(\theta^2)
,
\end{align*}
where $H_N=\ddel E_{0,N}(u^*_N)$ and $H_\infty=\ddel E_\infty(u^*_\infty)$.

This quantity can be estimated further.
However, for the purpose of illustrating the proposed theory without too many technicalities,  we simplify the problem and assume that $u^*_\infty=0$, $u^*_N=0$, and $\calP(u)\equiv 0$.
That is, we test how well $\mu_{\theta,N}$ approximates $\mu_{\theta,\infty}$ on a uniform lattice (without defects).
The advantage of this simplification is that now the Hessians can be inverted explicitly.
We will hence be able to explicitly track the leading term in all our estimates, which will automatically ensure that our estimates are sharp.

The error estimate is hence rewritten as
\begin{align*}
\<\mu_{\theta,N} - \mu_{\theta,\infty}, A\>
=~&
	\theta \Big(
		\smfrac12 \ddel A(0) \,:\, H_N^{-1}
	-
	\del A(0) \cdot \big(
	H_N^{-1} \big(\del E^{(1)}(0) + \smfrac12 H_N^{-1} \,:\, \dddel E^{(0)}(0)\big)
\big)
\\~&
	-
		\smfrac12 \ddel A(0) \,:\, H_\infty^{-1}
	+
	\del A(0) \cdot \big(
	H_\infty^{-1} \big(\del E^{(1)}(0) + \smfrac12 H_\infty^{-1} \,:\, \dddel E^{(0)}(0)\big)
\big)\Big)
\\~&+
	O(\theta^2)
.
\end{align*}

In what follows, we omit most of the details of the calculations, which essentially reduce to summation of geometric sequences.

\subsection{Free BCs}

We interpret the operators as matrices in the canonical basis of $\calU_\infty = \bbR^{\calB_\infty}$.
In matrix notation, we have that
\[
(H_N^{-1})_{ij}
=
\begin{cases}
	\frac{
		\lambda ^{-i-j}
		\left(\lambda ^{2 i}+\lambda ^{2 N}\right)
		\left(\lambda^{2 (j+N)}+1\right)
	}{
		2 \kappa \left(1-\lambda ^{4 N}\right)
	} & j\leq i, \\
	\frac{
		\lambda ^{-i-j}
		\left(\lambda ^{2 j}+\lambda ^{2 N}\right)
		\left(\lambda^{2 (i+N)}+1\right)
	}{
		2 \kappa \left(1-\lambda ^{4 N}\right)
	} & j> i,
\end{cases}
\]
where $\lambda = -\frac{1-\kappa}{\alpha}$ (see \eqref{eq:alpha_def} for the definition of $\alpha$).

We then have that
\begin{align*}
\<\mu_{\theta,N}, A\>
=~&
	A(0)
	+ \smfrac12 \theta A''(0) \smfrac{1+\lambda^{2N}}{2\kappa (1-\lambda^{2N})}
	-
	\smfrac12 \theta A'(0)
	V_{xxx}(0,0)
	 \big(
		-\smfrac{1+\lambda}{2 \kappa^2 ( 1-\lambda)}
		+
		\lambda^N \, \smfrac{1+\lambda+\lambda^2+\lambda^3}{\kappa^2 ( 1-\lambda^3)}
		+ O(\lambda^{2N})
	\big)
\\~&
	-
	\smfrac12 \theta A'(0)
	V_{xxy}(0,0)
	 \big(
		\smfrac{(1+\lambda)(1+2\lambda)}{2 \kappa^2 (1-\lambda)}
		+
		\lambda^{N-1} \,
		\smfrac{(1+\lambda)
			(1+2\lambda+6\lambda^2+2\lambda^3+\lambda^4)}{2 \kappa ^2 (1-\lambda^3)}
		+ O(\lambda^{2N-1})
	\big)
	+
	O(\theta^2)
,
\end{align*}
where we assumed that $A(u) = A(u^*_0)$, i.e., is supported only on the central bond, $u^*_0$.

We note that the terms with $\lambda^{N}$ and $\lambda^{N-1}$ are related to the error induced by the boundary condition (decaying exponentially in this case) and $O(\lambda^{2N})$ and $O(\lambda^{2N-1})$ are the higher-order terms in this error.

Thus,
\begin{align*}
\<\mu_{\theta,N}-\mu_{\theta,\infty}, A\>
=~&
	-
	\smfrac12 \theta A'(0) \big[
	V_{xxx}(0,0)
		\lambda^N \, \err^{\rm fr}_1
	+
	V_{xxy}(0,0)
		\lambda^{N-1} \, \err^{\rm fr}_2
	\big]
	+
	O(\theta^2 + \theta \lambda^{2N-1})
,
\end{align*}
where
\[
\err^{\rm fr}_1 =
	\frac{1+\lambda+\lambda^2+\lambda^3}{\kappa^2 ( 1-\lambda^3)}
\qquad\text{and}\qquad
\err^{\rm fr}_2 =
		\frac{(1+\lambda)
			(1+2\lambda+6\lambda^2+2\lambda^3+\lambda^4)}{2 \kappa ^2 (1-\lambda^3)}
\]
are the respective error coefficients for the free boundary conditions.

\subsection{A Quasicontinuum Method}\label{sec:applications:QC}

\subsubsection*{Derivation}

We consider a system of size $M$ with free boundary conditions, but we want to eliminate (that is, integrate out) all the degrees of freedom outside the subsystem of size $N$ ($N\ll M$) by defining the following free energy \cite{TadmorLegollKimEtAl2013}
\begin{equation}\label{eq:applications:QC:Vbc_der1}
\hat{V}^\bc_{\theta}(u_N)
=
-\theta \log\bigg(
\int\cdots\int e^{-\theta^{-1}\sum_{\xi=N+\half}^{M-\half} V(u_{\xi-\half}, u_{\xi+\half} )} \dd u_{N+1}\ldots \dd u_{M}
\bigg).
\end{equation}
By doing this we reduced the system size, but we have not yet computationally simplified the problem.
Our first step in simplifying \eqref{eq:applications:QC:Vbc_der1} will be a so-called quasiharmonic approximation to $\hat{V}^\bc_\theta(u_N)$,
\begin{equation}\label{eq:applications:QC:Vbc_der2}
\tilde{V}^\bc_{u^\dagger,\theta}(u_N)
=
-\theta \log\bigg(
\int\cdots\int e^{-\theta^{-1} V(u_{N}, u_{N+1}) - \theta^{-1}\sum_{\xi=N+\threehalfs}^{M-\half} \tilde{V}^\lin_{u^\dagger,\xi}(u_{\xi-\half}, u_{\xi+\half})} \dd u_{N+1}\ldots \dd u_{M}
\bigg),
\end{equation}
where
\begin{align*}
\tilde{V}^\lin_{u^\dagger,\xi}(u^\dagger_{\xi-\half}+x,u^\dagger_{\xi+\half}+y)
=~& V(u^\dagger_{\xi-\half},u^\dagger_{\xi+\half}) + \<\del V(u^\dagger_{\xi-\half},u^\dagger_{\xi+\half}), (x,y)\>
\\~&+\smfrac12 \<\ddel V(u^\dagger_{\xi-\half},u^\dagger_{\xi+\half}) (x,y), (x,y)\>,
\end{align*}
and $u^\dagger\in\bbR^{\{N+1,\ldots,M\}}$ should be chosen as the minimizer of $\tilde{V}^\bc_{u^\dagger,\theta}(u_N)$ (which hence depends on $\theta$ and $u_N$).

We next restrict $u^\dagger_{N+1} = u^\dagger_{N+2} = \ldots = u^\dagger_{M} =: \mF$ in the spirit of finite elements (note that piecewise linear elements for the displacements corresponds to piecewise constant elements for the strain $u^\dagger$).
This defines $\ttilde{V}^{\bc}_{\mF,\theta}$ by a formula similar to \eqref{eq:applications:QC:Vbc_der2} with
\begin{align*}
\ttilde{V}^\lin_{\mF}(\mF+x,\mF+y)
=~& V(\mF,\mF) + \<\del V(\mF,\mF), (x,y)\>
+\smfrac12 \<\ddel V(\mF,\mF) (x,y), (x,y)\>.
\end{align*}
Note that
$\ttilde{V}^\bc_{\mF,\theta}(u_N)$ can be computed explicitly as
\begin{equation}\label{eq:applications:QC:Vbc_der3}
\ttilde{V}^\bc_{\mF,\theta}(u_N)
=
V(u_{N}, \mF)+
\sum_{\xi=N+\threehalfs}^{M-\half} V(\mF,\mF)
+ \frac{\theta}{2}\, \log(\det H_{u^\dagger,\mF,N,M}),
\end{equation}
where $H_{u^\dagger,\mF,N,M}$ can be thought of as an $(M-N)\times(M-N)$ matrix and is equal to a submatrix of $\ddel E_{0,M}(\mF)$ corresponding to the degrees of freedom given by $u_{N+1}, \ldots, u_{M}$.

The next step is to change the energy involving the last two terms of \eqref{eq:applications:QC:Vbc_der3} by the continuum free energy:
\[
V^\bc_{\mF,\theta}(u_N)
=
V(u_{N}, \mF)+
(M-N+1) W^{\free}_\theta(\mF),
\]
where $\mF$ is now chosen as the minimizer of $V^\bc_{\mF,\theta}$, and the continuum free energy density is
\begin{equation}\label{eq:Wfree}
W^\free_\theta(\mF)
:=
V(\mF,\mF)+\frac{\theta}{2} \, \log\Big(\frac{1}{2\pi\theta}\Big)
-\frac{\theta}{2} \,
\log \smfrac{V_{xx}(\mF,\mF)-\sqrt{(V_{xx}(\mF,\mF))^2-(V_{xy}(\mF,\mF))^2}}{(V_{xy}(\mF,\mF))^2},
\end{equation}
which is derived in Appendix \ref{sec:Wfree}.

In the limit $M-N\to\infty$, $\mF$ will just be the minimizer of $W^\free_\theta(\mF)$, which can be computed explicitly as
\[
\argmin(W^\free_\theta)
= \theta \big(
	-\smfrac{1}{4 (1+\alpha)\kappa} V_{xxx}(0,0)
	-\smfrac{2-2\kappa-\alpha}{4 (1+\alpha)\kappa} V_{xxy}(0,0)
\big)
+ O(\theta^2),
\]
where $\alpha = V_{xy}(0,0)$ and $\kappa=\sqrt{1-\alpha^2}$ (cf.~\eqref{eq:alpha_def}),
and hence
\begin{equation}\label{eq:applications:QC:Vbc}
V^\bc_\theta(x) =
\underbrace{V(x, 0)}_{=:V^{\bc,(0)}(x)}
+
\theta
\underbrace{V_y(x, 0)
\big(
	-\smfrac{1}{4 (1+\alpha)\kappa} V_{xxx}(0,0)
	-\smfrac{2-2\kappa-\alpha}{4 (1+\alpha)\kappa} V_{xxy}(0,0)
\big)}_{=:V^{\bc,(1)}(x)}
+ O(\theta^2),
\end{equation}
where we ignored the constant term.

\subsubsection*{Error Estimate}

Similar to the free BCs case, explicit calculations can be performed with the boundary conditions given by \eqref{eq:applications:QC:Vbc} and we can find that
\begin{align*}
\<\mu_{\theta,N}-\mu_{\theta,\infty}, A\>
=~&
	-
	\smfrac12 \theta A'(0) \big[
	V_{xxx}(0,0)
		\lambda^N \, \err^{\rm qc}_1
	+
	V_{xxy}(0,0)
		\lambda^{N-1} \, \err^{\rm qc}_2
	\big]
	+
	O(\theta^2 + \theta \lambda^{2N-1})
,
\end{align*}
where
\[
\err^{\rm qc}_1 =
	\frac{\lambda ^2 + \lambda^3}{\kappa ^2
	   (1-\lambda ^3)}
\qquad\text{and}\qquad
\err^{\rm qc}_2 =
		\frac{(1+\lambda)
			(\lambda+2\lambda^2+2\lambda^4+\lambda^5)}{2 \kappa ^2 (1-\lambda^3)}
\]
are the respective error coefficients for the QC method.

\subsection{Comparison}

First, we notice that the error for the QC is always smaller than that of the free BCs, i.e., $|\err^{\rm qc}_1|<|\err^{\rm fr}_1|$ and $|\err^{\rm qc}_2|<|\err^{\rm fr}_2|$.
This can be easily proved by comparing the respective polynomials in $\lambda$.

Second, by plotting the graphs of $\err_1$ and $\err_2$ for the two methods, see Figure \ref{fig:err}, we can see that the leading order error of the two methods differ by about an order of magnitude (unless the lattice is close the point of instability, $\alpha=\pm 1$).

\begin{figure}
\begin{center}
\hfill
\subfigure[]{\label{fig:err:1}\includegraphics{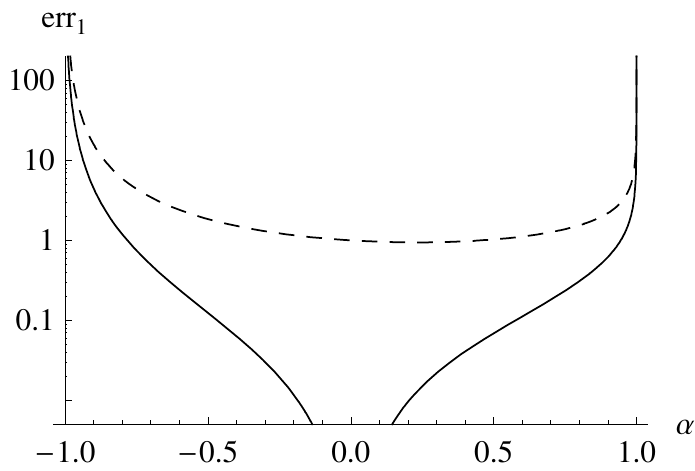}}
\hfill\hfill
\subfigure[]{\label{fig:err:2}\includegraphics{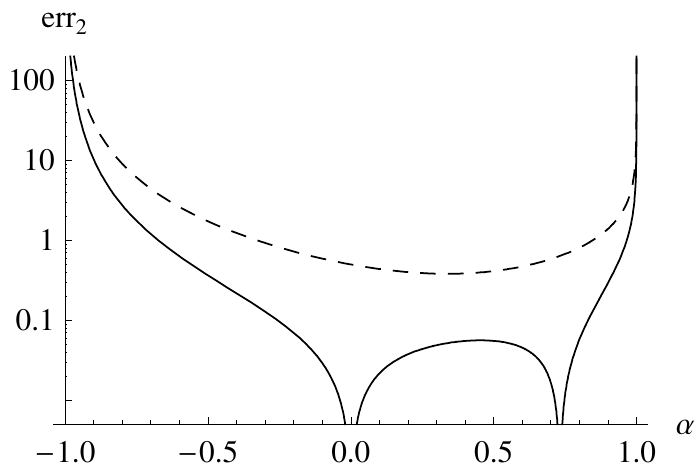}}
\hfill$\mathstrut$
\caption{Error coefficients $\err_1$ and $\err_2$ of the free BCs (dashed) and the QC method (solid). One can see that the leading order error for the QC is typically one or two orders of magnitude less than that of the free BCs (unless the lattice is close the point of instability, $\alpha=\pm 1$).}
\label{fig:err}
\end{center}
\end{figure}

\section{Proof for the Finite Lattice}\label{sec:finite_lattice}

In this section we prove the statements of Theorem \ref{th:uniform_expansion} for the case $N<\infty$, namely parts (a), (c, $N<\infty$), and (d, $N<\infty$).

In the proofs $C$ will denote a generic constant which is independent of $N$ and $\theta$.

\subsection{Zero Temperature Equilibrium}

The exact equilibrium $u_\infty^*$ can be approximated on a finite lattice by $u_N^*\in\ell^2(\calU_N)$ satisfying
\begin{equation}\label{eq:defect-N-def}
\<\del E_{0,N}(u_N^*), v\> = 0,
\qquad \forall v\in\calU_N
\end{equation}
and additionally requiring that $u_N^*$ is sufficiently close to $u^*$.
Lemma \ref{lem:Enk_convex} below implies that $u_N^*$ is also a global minimum of $E_{0,N}$ for $N$ large enough, which makes it consistent with the definition of $u_N^*$ made in Theorem \ref{th:uniform_expansion}(d).

\begin{lemma}\label{lem:zero_temp_approx}
There exists $\rho>0$ such that for $N$ large enough there is a unique solution $u_N^*$ to \eqref{eq:defect-N-def} within the ball $\|u^*_{N}-u^*|_{\calB_N}\|_{\ell^2}\leq \rho$. Furthermore, $\|u^*_{N}-u^*|_{\calB_N}\|_{\ell^2}\to 0$ as $N\to\infty$.
\end{lemma}
\begin{proof}
The standard consistency-stability argument will be applied.

\noindent {\it Uniform regularity of $E_{0,N}$.}
The uniform $\C^3$ regularity of $E_{0,N}(u)$ for $u\in B_{r_V}$ follows from, by now, the standard arguments (see, e.g., \cite{LuskinOrtner2013acta,LiOrtnerShapeevEtAl2014,EhrlacherOrtnerShapeev2013preprint,OrtnerTheil2013}).
Indeed, $\calP(u)$ and $V^\bc_0$ are $\C^3$ functions of $u\in\ell^2(\calL_N)$ and do not depend on $N$, and the rest of the energy can be estimated using the Taylor series around $0$:
\[
	T(u) :=
	\sum_{\xi\in\calL_N} V(u_{\xi-\half}, u_{\xi+\half})
	=
	\sum_{\xi\in\calL_N} \<\ddel V(\eta_\xi u_{\xi-\half}, \eta_\xi u_{\xi+\half})(u_{\xi-\half}, u_{\xi+\half}),(u_{\xi-\half}, u_{\xi+\half})\>.
\]
where we used \eqref{eq:orig_ass:Vzero} and $\del V(0,0)=0$ (the latter follows from \eqref{eq:orig_ass:V_growth}) and where $\eta_\xi$ are some numbers between $0$ and $1$.
By choosing $\rho<r_V-\|u^*\|_{\ell^\infty}$, we can estimate $\ddel V(\eta_\xi u_{\xi-\half}, \eta_\xi u_{\xi+\half}) \leq c_V$ from \eqref{eq:orig_ass:V_bound} which immediately implies $|T(u)| \leq 2 c_V \|u\|_{\ell^2}^2$.
Similar arguments can be used to bound $|\del^k T(u)|$ by $\|u\|_{\ell^2}$ for $k=1,2,3$.

\noindent {\it Consistency.}
The consistency error can be easily estimated as follows:
\begin{align*}
\<\del E_{0,N}(u^*|_{\calB_n}), v\>
=~&
	v_N \big(
		V^\bc_x(u^*_{\infty,N}) -
		V_x(u^*_{\infty,N}, u^*_{\infty,N+1})
	\big)
\\~&+
	v_{-N} \big(
		V^\bc_x(u^*_{\infty,-N}) -
		V_y(u^*_{\infty,-N-1}, u^*_{\infty,-N})
	\big)
\\ \leq~& \|v\|_{\ell^2} C (|u^*_{\infty,N}| + |u^*_{\infty,N+1}|)
\to 0
\qquad\text{as $N\to\infty$}
,
\end{align*}
where we used \eqref{eq:orig_ass:Vbc_force} in the transition to the last line.

\medskip \noindent {\it Stability.}
The stability,
\begin{equation}\label{eq:defect_stab}
\<\ddel E_{0,N}(u^*|_{\calB_N}) v, v\>
\geq
(\min\{\gamma_E,\gamma_V\}+o(1)) \|v\|_{\ell^2}^2
\qquad \forall v\in\ell^2(\bbZ)
\end{equation}
is less straightforward, yet follows from standard concentration compactness arguments.
Indeed, assume that \eqref{eq:defect_stab} is false: that there exist $N_i \to \infty$, $v_i\in\calU_N$, $\|v_i\|=1$, such that
\[
\lim_{i\to\infty}\<\ddel E_{0,N_i}(u^*|_{\calB_{N_i}}) v_i, v_i\> < \min\{\gamma_E,\gamma_V\}.
\]
We then can extract subsequences, which we still denote as $N_i$ and $v_i$, such that $v_i\weakto v_\infty\in\ell^2(\bbZ)$, where we implicitly extend $v_i$ by zero outside $\calB_{N_i}$.
Since a weak convergence implies a pointwise convergence, we have that
$v_{i,j}\to v_{\infty,j}$ for any fixed $j$.
This allows one to choose a sequence $r_i\to\infty$, $r_i < i$, such that if we define
\[
\bar{v}_{i,j} := \begin{cases}
v_{i,j} & |j| \leq r_i \\
0 & |j| > r_i, \\
\end{cases}
\]
and $\tilde{v}_i = v_i - \bar{v}_i$ then $\bar{v}_i \to v_\infty$ strongly in $\ell^2$ and $\tilde{v}_i \weakto 0$ \cite{EhrlacherOrtnerShapeev2013preprint}.

We then will use
\begin{itemize}
\item[(i)]
$
\<\ddel E_{\infty}(u^*|_{\calB_{N_i}}) \bar{v}_i, \bar{v}_i\>
\geq
\gamma_E \|\bar{v}\|_{\ell^2}^2
$
for all $v\in\ell^2(\bbZ)$
which follows from \eqref{eq:defect-global},

\item[(ii)]
$
\<\ddel
	V(u^*_{\xi-\half}, u^*_{\xi+\half})
	(\tilde{v}_{i,\xi-\half}, \tilde{v}_{i,\xi+\half})
	,
	(\tilde{v}_{i,\xi-\half}, \tilde{v}_{i,\xi+\half})
\>
\geq
(\gamma_V + o(1)) (\tilde{v}_{\xi-\half}^2+\tilde{v}_{\xi+\half}^2)
$
as $\xi\to\pm\infty$ which follow from \eqref{eq:orig_ass:V_growth} and the fact that $u^*_{\xi\pm\half}\to 0$ as $\xi\to\pm\infty$,
and lastly
\item[(iii)]
that the scalar product $(\bar{v}_i, \tilde{v}_i)_{\ell^2}\to (v_\infty,0)_{\ell^2} =0$ and likewise $\<\ddel E_{0,N_i}(u^*|_{\calB_{N_i}}) \bar{v}_i, \tilde{v}_i\> \to 0$
\end{itemize}
to estimate
\begin{align*}
\<\ddel E_{0,N_i}(u^*|_{\calB_{N_i}}) v_i, v_i\>
=~&
\<\ddel E_{0,N_i}(u^*|_{\calB_{N_i}}) \bar{v}_i, \bar{v}_i\>
+
\<\ddel E_{0,N_i}(u^*|_{\calB_{N_i}}) \tilde{v}_i, \tilde{v}_i\>
+ o(1)
\\~&
=
\<\ddel E_{\infty}(u^*) \bar{v}_i, \bar{v}_i\>
+
\<\ddel E_{0,N_i}(u^*|_{\calB_{N_i}}) \tilde{v}_i, \tilde{v}_i\>
+ o(1)
\\~&
\geq
\<\ddel E_{\infty}(u^*|_{\calB_{N_i}}) \bar{v}_i, \bar{v}_i\>
+
\sum_{|\xi|>r_i} \smfrac{\gamma_V}{2} (\tilde{v}_{\xi-\half}^2+\tilde{v}_{\xi+\half}^2)
+ o(1)
\\~&
\geq
\gamma_E \|\bar{v}_i\|^2
+
\gamma_V \|\tilde{v}_i\|^2
+ o(1)
\\~&
\geq
\min\{\gamma_E, \gamma_V\} (\|v_i\|^2 - 2 (\bar{v}_i, \tilde{v}_i))
+ o(1)
\\~&
=
\min\{\gamma_E, \gamma_V\}
+ o(1),
\end{align*}
which is a contradiction.

\medskip \noindent {\it Finalizing.}
This allows us to apply the implicit function theorem (IFT) (see, e.g., \cite{EhrlacherOrtnerShapeev2013preprint, LiOrtnerShapeevEtAl2014} for examples of such application of the IFT to quantifying the accuracy of defects calculation) and the assertion of this theorem follows.
\end{proof}
A trivial, but important corollary is the uniform decay at infinity of $|u^*_{N,n}|$:
\begin{corollary}
$\sup_{N\in\bbN, N\geq n}|u^*_{N,n}|\to 0$ as $n\to\infty$.
\end{corollary}

\subsection{Change of Variables}\label{sec:finite_lattice:change}

For the purpose of our finite-temperature analysis, we first make a change of variables, which essentially (i) shifts the energy minimizer, $u_N^*$, to the origin, (ii) fixes the additive constant in the energy to be zero at $u_N^*$, (iii) subtracts the dead-load forces to yield $\nabla\tilde{V}_{N,\xi}(0,0)=0$:
\begin{align*}
	\tilde \calP_{N}(v) :=~&
	\calP(u^*_{N}+v)
	- \calP(u^*_{N})
	- \<\del\calP(u^*_{N}), v\>
\\
	\tilde V_{N,\xi}(v_{\xi-\half},v_{\xi+\half}) :=~& V(u^*_{N,\xi-\half}+v_{\xi-\half},u^*_{N,\xi+\half}+v_{\xi+\half})
	\\~&
	-\delta_{N-|\xi|-\half} \smfrac{\gamma_V}{2} \smfrac{v_{N,\xi+\half\sgn(\xi)}^2}{2}
	-V(u^*_{N,\xi-\half},u^*_{N,\xi+\half})
	\\~&
	-v_{\xi-\half} V_x(u^*_{N,\xi-\half},u^*_{N,\xi+\half})
	-v_{\xi+\half} V_y(u^*_{N,\xi-\half},u^*_{N,\xi+\half})
\\
	\tilde V_{\theta,N}^\bc(x)
	:=~&
	\smfrac{\gamma_V}{2} \smfrac{x^2}{2}
	+
	V^\bc_\theta(u^*_{N,N}+x)
	-
	V^\bc_{\theta}(u^*_{N,N})
	-
	x (V^{\bc,{(0)}})_{x}(u^*_{N,N})
	,
\\
	\tilde E_{\theta,N}(v)
:=~&
	\tilde \calP_{N}(v) +
	\sum_{\xi\in\calL_N} \tilde V_{N,\xi}(v_{\xi-\half},v_{\xi+\half})
	+ \tilde V_{0,N}^{\bc}(v_N)
	+ \tilde V_{0,N}^{\bc}(v_{-N}),
\end{align*}
where $\delta_\bullet$ is the Kronecker delta.
Notice that we also added $\smfrac{\gamma_V}{2} \smfrac{x^2}{2}$ to $\tilde{V}^\bc$ which was then subtracted from $\tilde{V}_{N,N-\half}$ and $\tilde{V}_{N,-N+\half}$.
This change of variables essentially does not change the original energy, as we can see from the following proposition.
\begin{proposition}
$
\tilde E_{\theta,N}(v) = E_{\theta,N}(u_N^* + v) - E_{\theta,N}(u_N^*).
$
\end{proposition}
\begin{proof}
We can easily write the difference as
\begin{align*}
&
	E_{\theta,N}(u_N^* + v) - E_{\theta,N}(u_N^*) - \tilde E_{\theta,N}(v)
\\=~&
	\<\del \calP(u^*_N), v\> +
	\sum_{\xi\in\calL_N} \big(
		v_{\xi-\half} V_x(u^*_{N,\xi-\half},u^*_{N,\xi+\half})
		+
		v_{\xi+\half} V_y(u^*_{N,\xi-\half},u^*_{N,\xi+\half})
	\big)
\\~&
	+ v_{N} (V^{\bc,(0)}_{N})_x(u^*_{N,N})
	+ v_{-N} (V^{\bc,(0)}_{N})_x(u^*_{N,-N})
\\=~& \<\del E_{0,N}(u_N^*), v\> = 0.
\end{align*}
\end{proof}

Recall that we assumed $u^*_{-j}=u^*_j$, which implies $\tilde{V}_{N,\xi} = \tilde{V}_{N,-\xi}$.

We next list and prove the properties of the newly introduced functions that will effectively be used as assumptions for this section. (In other words, we will establish uniform in $N$ and $\theta$ expansion of the Gibbs measures based on the assumptions below on $\tilde{V}_{N,\xi}$, etc., regardless of the origin of such functions.)

\begin{lemma}\label{lem:ass}
For any $\eps>0$ there exist $K_0, N_0\in\bbN$ such that for $N \geq N_0$ with $\tilde{\gamma}_V=\gamma_V/2-\eps$, $\tilde{\gamma}_\bc=\gamma_V/2$, $\tilde{\gamma}_E=\min\{\gamma_E,\gamma_V\}-\eps$, and some positive $\tilde{c}_\rho$, $\tilde{c}_\calP$, $\tilde{c}_V$, and $\tilde{r}_V$ the following properties hold for all $\xi\in\calL_N$, $\theta>0$, and $m=0,1,\ldots,M-1$:
\begin{align} \notag
&
\tilde{\calP}_N, \tilde{V}_{N,\xi}^{(m)}, \tilde{V}_{N}^{\bc,(m)}, \tilde{V}_{\theta,N}^{\bc,(M^+)} \text{ are measurable with values in $\overline{\bbR}$}
\\ \label{eq:ass:P_bound}
& \tilde{\calP}_N(v)
\geq
-\tilde{c}_\calP (1+\|v\|),
\\ \label{eq:ass:P_bound_loc}
&
\|\tilde{\calP}_{N}\|_{\C^{2M+3}(B_{\tilde{r}_V})} \leq \tilde{c}_V,
\quad
\\ \label{eq:ass:V_bound}
&
\|\tilde{V}_{N,\xi}\|_{\C^{2M-2 m+3}(B_{\tilde{r}_V})} \leq \tilde{c}_V,
\\ \label{eq:ass:Vbc_bound}
&
\|\tilde{V}_{N}^{\bc,(m)}\|_{\C^{2M-2 m+3}(B_{\tilde{r}_V})} \leq \tilde{c}_V,
\quad
\|\tilde{V}_{\theta,N}^{\bc,(M^+)}\|_{C(B_{\tilde{r}_V})} \leq \tilde{c}_V
,
\\ \label{eq:ass:V_zero_order}
&
\tilde{V}_{N,\xi}^{(0)}(0,0) = 0
,
\\ \notag
&
\tilde{V}_{N}^{\bc,(0)}(0) = 0
.
\end{align}

Moreover, the following properties hold for all $N \geq N_0$, $|\xi| \geq K_0$, $\theta>0$, and $(x,y)\in\bbR^2$:
\begin{align}
\label{eq:ass:V_growth}
&
\tilde{V}_{N,\xi}(x,y) \geq \smfrac{\tilde{\gamma}_V}{2} (x^2+y^2),
\\ \label{eq:ass:Vbc_growth}
&
\tilde{V}_{N}^{\bc,(0)}(x) \geq \smfrac{\tilde{\gamma}_\bc}{2} x^2,
\quad
\tilde{V}_{N}^{\bc,(m)}(x) \geq -c_V,
\quad
\tilde{V}_{\theta,N}^{\bc,(M^+)}(x) \geq -c_V,
\\ \label{eq:ass:E_convex}
&
\tilde{E}_{0,N}(v) \geq \smfrac{\tilde{\gamma}_E}{2} \|v\|^2
.
\end{align}

Finally, the following holds:
\begin{align}
\label{eq:ass:V_asymptote}
&
\left\{
\begin{array}{l}
	(\tilde{V}_{N,\xi})_{xx}(0) \to 1 \\
	(\tilde{V}_{N,\xi})_{xy}(0) \to \alpha \\
	(\tilde{V}_{N,\xi})_{yy}(0) \to 1-\delta_{N-\xi-\half}\,\frac{\gamma_V}{2}
\end{array}
\right.
\text{~~~as $\xi,N\to\infty$}
\end{align}
(cf.\ \eqref{eq:alpha_def} for the definition of $\alpha$).
We note that, here and below, by $\xi,N\to\infty$ we mean that $\xi\to+\infty$ and $N\to+\infty$ such that $0\leq\xi<N$ but otherwise independently of each other, unless additional constraints are given. (The case $\xi\to-\infty$ will be taken care of by the symmetricity $\tilde{V}_{N,\xi} = \tilde{V}_{N,-\xi}$.)
\end{lemma}

We first note that all the statements except \eqref{eq:ass:V_growth}, \eqref{eq:ass:Vbc_growth}, and \eqref{eq:ass:E_convex}, easily follow from the definition of $\tilde{V}_{N,\xi}$ and $\tilde{V}^\bc_{\theta,N}$ and the convergence of $u^*_N$ to $u^*_\infty$.
Note that we need to choose $\tilde{r}_V \leq r_V - \|u^*\|_{\ell^\infty}$ in order to satisfy \eqref{eq:ass:P_bound_loc}--\eqref{eq:ass:Vbc_bound}.

Furthermore, \eqref{eq:ass:V_growth} and \eqref{eq:ass:Vbc_growth} follow easily from the following Lemma.

\begin{lemma}\label{lem:quad_growth}
Let $X$ be a Banach space, $E:X\to\overline{\bbR}$, be such that $E\in \C^3(B_r)$, $E(0)=0$, and $E(v)\geq \frac{\gamma}{2} \|v\|^2$ for all $v\in X$.
Then for any $\eps>0$ there exists $R>0$ such that for all $\|u\|\leq R$ and $v\in X$
\begin{align}\label{eq:quad_growth}
E(u+v)-E(u)-\<\del E(u), v\>\geq \smfrac{\gamma-\eps}{2} \, \|v\|^2
.
\end{align}
\end{lemma}
\begin{proof}

\noindent{\it Proof of \eqref{eq:quad_growth}.}
Fix $\eps>0$.
First, note that $\ddel E(w) \geq \gamma-\eps$ in the sense of self-adjoint operators, when $\|w\|\leq 2\rho$ for some $\rho$, $0<2\rho\leq r$.
The Taylor expansion, $E(u+v)-E(u)-\<\del E(u), v\> = \frac12 \<\ddel E(u+v\eta_{u,v}) v, v\>$, where $\eta_{u,v}\in(0,1)$, hence proves \eqref{eq:quad_growth} for $\|v\|\leq \rho$, once we choose $R\leq\rho$.

Then, notice that $\del E(u) = O(\|u\|) \leq O(R)$ and $E(u) = O(\|u\|^2) \leq O(R^2)$. Hence for $\|v\|>\rho$, we estimate in the limit $R\to 0$:
\begin{align*}
E(u+v)-E(u)-\<\del E(u), v\>
\geq~&
	\smfrac{\gamma}{2} \|u+v\|^2
	- E(u) - \|\del E(u)\| \|v\|
\\=~&
	\smfrac{\gamma}{2} \|v\|^2
	- O(\|u\|^2) - O(\|u\|) \|v\|
\\ \geq ~&
	\smfrac{\gamma}{2} \|v\|^2
	- O(R^2) \|v\|^2/\rho^2 - O(R) \|v\|^2/\rho
\\ =~&
	\smfrac{\gamma}{2} (1+O(R/\rho))\,\|v\|^2
.
\end{align*}
Hence \eqref{eq:quad_growth} follows once $R$ is chosen small enough.
\end{proof}

Finally, \eqref{eq:ass:E_convex} follows from the following, more general lemma, applied with $K=N$.

\begin{lemma}\label{lem:Enk_convex}
Let $\tilde{E}_{\theta,N,K} : \calU_K \to \bbR$, be defined by
\[
\tilde{E}_{\theta,N,K}
:=
	\tilde \calP_{N}(v) +
	\sum_{\xi\in\calL_K} \tilde V_{N,\xi}(v_{\xi-\half},v_{\xi+\half})
	+ \tilde V_{\theta,N}^{\bc}(v_K)
	+ \tilde V_{\theta,N}^{\bc}(v_{-K})
\]
for $K \leq N$.
Then $\tilde{E}_{0,N,K}(v) \geq (\tilde{\gamma}_E - o(1)) \frac{\|v\|^2}{2}$ as $K\to\infty$.
\end{lemma}
\begin{proof}
Assume the converse, that there exist $k_i, n_i\in\bbN$, $n_i \geq k_i\to\infty$, and $v_i\in\calU_{n_i}$, $v_i\ne 0$, such that
\[
\liminf_{i\to\infty}\frac{\tilde{E}_{0,n_i,k_i}(v_i)}{\|v_i\|^2}
=
\liminf_{i\to\infty}\frac{E_{0,k_i}(u^*_{n_i} + v_i)- E_{0,k_i}(u^*_{n_i})}{\|v_i\|^2}
 < \smfrac12 \min\{\gamma_E,\gamma_V\}.
\]
We can extract a subsequence of $v_i$ (which we still denote by $v_i$) such that $\lim_{i\to\infty}\|v_i\| \in \bbR^+\cup\{+\infty\}$.

\medskip \noindent {\it Case 1 ($0<\lim_{i\to\infty}\|v_i\|<+\infty$).}
In this case we use the concentration compactness argument.
We further extract a subsequence (again not relabeled) such that $v_i \weakto v_\infty \in \ell^2(\bbZ)$, where we implicitly extend $v_i$ by zero outside $\calB_{n_i}$.
Then, similarly as in the proof of Lemma \ref{lem:zero_temp_approx}, there exists a sequence $r_i\to\infty$, $r_i<k_i$, such that if we define $\bar{v}_i \in \ell^2(\bbZ)$,
\[
\bar{v}_{i,j} := \begin{cases}
v_{i,j} & |j| < r_i \\
0 & |j| \geq r_i, \\
\end{cases}
\]
then $\bar{v}_i \to v_\infty$ strongly in $\ell^2$ and $\tilde{v}_i := v_i - \bar{v}_i \weakto 0$.
Then one can express
\begin{align*}
\tilde{E}_{0,n_i,k_i}(v_i)
=~&
E_{0,k_i}(u^*_{n_i} + v_i) - E_{0,k_i}(u^*_{n_i})
\\ =~&
		\big(\calP(u^*_{n_i} + v_i) - \calP(u^*_{n_i})\big)
	\\ &+
		\sum_{\substack{\xi\in\calL_{k_i} \\ |\xi| < r_i-\half}} \big(
		V(u^*_{n_i, \xi-\half} + \bar{v}_{i,\xi-\half}, u^*_{n_i, \xi+\half} + \bar{v}_{i,\xi+\half})
		- V(u^*_{n_i, \xi-\half}, u^*_{n_i, \xi+\half})
		\big)
	\\ &+
		\big(
		V(u^*_{n_i, r_i-1} + \bar{v}_{i,r_i-1}, u^*_{n_i, r_i} + \tilde{v}_{i,r_i})
		- V(u^*_{n_i, r_i-1}, u^*_{n_i, r_i})
		\big)
	\\ &+
		\big(
		V(u^*_{n_i, -r_i+1} + \bar{v}_{i,-r_i+1}, u^*_{n_i, -r_i} + \tilde{v}_{i,-r_i})
		- V(u^*_{n_i, -r_i+1}, u^*_{n_i, -r_i})
		\big)
	\\ &+
		\sum_{\substack{\xi\in\calL_{k_i} \\ |\xi| > r_i-\half}} \big(
		V(u^*_{n_i, \xi-\half} + \tilde{v}_{i,\xi-\half}, u^*_{n_i, \xi+\half} + \tilde{v}_{i,\xi+\half})
		- V(u^*_{n_i, \xi-\half}, u^*_{n_i, \xi+\half})
		\big)
	\\ &+
		V^\bc_{0,n_i}(u^*_{n_i, k_i} + \tilde{v}_{i,k_i})
		- V^\bc_{0,n_i}(u^*_{n_i, k_i})
		+ V^\bc_{0,n_i}(u^*_{n_i, -k_i} + \tilde{v}_{i,-k_i})
		- V^\bc_{0,n_i}(u^*_{n_i, -k_i})
.
\end{align*}
Note that here we cannot use that, e.g., $E_{0,k_i}(u^*_{n_i} + v_i) = E_{0,k_i}(u^*_{\infty} + v_i) + o(1)$ since the regularity of $E$ is assumed only in a finite neighborhood of $u^*_{\infty}$ while $\|v_i\|$ may not be too large.

We next use that $\calP(u^*_{n_i} + v_i) = \calP(u^*_{n_i} + \bar{v}_i)$ for $i$ large enough, and use the difference $E_\infty(u^*_{n_i} + \bar{v}_i) - E_\infty(u^*_{n_i})$ in the following way:
\begin{align*}
\tilde{E}_{0,n_i,k_i}(v_i)
=~&
		\big(E_\infty(u^*_{n_i} + \bar{v}_i) - E_\infty(u^*_{n_i})\big)
	\\ &+
		\big(
		V(u^*_{n_i, r_i-1} + \bar{v}_{i,r_i-1}, u^*_{n_i, r_i} + \tilde{v}_{i,r_i})
		- V(u^*_{n_i, r_i-1} + \bar{v}_{i,r_i-1}, u^*_{n_i, r_i})
		\big)
	\\ &+
		\big(
		V(u^*_{n_i, -r_i+1} + \bar{v}_{i,-r_i+1}, u^*_{n_i, -r_i} + \tilde{v}_{i,-r_i})
		- V(u^*_{n_i, -r_i+1} + \bar{v}_{i,-r_i+1}, u^*_{n_i, -r_i})
		\big)
	\\ &+
		\sum_{\substack{\xi\in\calL_{k_i} \\ |\xi| > r_i-\half}} \big(
		V(u^*_{n_i, \xi-\half} + \tilde{v}_{i,\xi-\half}, u^*_{n_i, \xi+\half} + \tilde{v}_{i,\xi+\half})
		- V(u^*_{n_i, \xi-\half}, u^*_{n_i, \xi+\half})
		\big)
	\\ &+
		V^\bc_{0,n_i}(u^*_{n_i, k_i} + \tilde{v}_{i,k_i})
		- V^\bc_{0,n_i}(u^*_{n_i, k_i})
		+ V^\bc_{0,n_i}(u^*_{n_i, -k_i} + \tilde{v}_{i,-k_i})
		- V^\bc_{0,n_i}(u^*_{n_i, -k_i})
.
\end{align*}

We estimate the first term using \eqref{eq:defect-global}:
\[
E_\infty(u^*_{n_i} + \bar{v}_i) - E_\infty(u^*_{n_i})
\geq
\smfrac{\gamma_E}{2} \|\bar{v}_i + u^*_{n_i}-u^*_{\infty}\|^2 + E_\infty(u^*_{\infty}) - E_\infty(u^*_{n_i})
=
\smfrac{\gamma_E}{2} \|\bar{v}_i\|^2 + o(1)
\]
as $i\to\infty$, where we use that $\|u^*_{n_i}-u^*_{\infty}\|\to 0$ and the continuity of $E_\infty$.
We next use
(i) the growth assumptions on the $V$ and $V^\bc$ functions,
(ii) that $\bar{v}_{i, \pm (r_i-1)} = v_{\infty, \pm (r_i-1)} + o(1)\to 0$,
(iii) the respective convergence properties of $u^*_{n_i}$, and
(iv) that $V(0,0)=0$, to likewise estimate the rest of the terms:
\begin{align*}
\tilde{E}_{0,n_i,k_i}(v_i)
\geq~&
		o(1)
		+
		\smfrac{\gamma_E}{2} \|\bar{v}_i\|^2
		+
		\smfrac{\gamma_V}{2} \big(|\tilde{v}_{i,r_i}|^2 + |\tilde{v}_{i,-r_i}|^2
		\big)
		+
		\smfrac{\gamma_V}{2}
		\sum_{\substack{\xi\in\calL_{k_i} \\ |\xi| > r_i-\half}}
		\big(|\tilde{v}_{i,\xi-\half}|^2 + |\tilde{v}_{i,\xi+\half}|^2
		\big)
\\ \geq~&
o(1) + \gamma_E \smfrac{\|\bar{v}_i \|^2}{2} + \gamma_V \smfrac{\|\tilde{v}_i\|^2}{2}
\\ \geq~&
\min\{\gamma_E, \gamma_V\} \smfrac{\| \bar{v}_i \|^2 + \| \tilde{v}_i \|^2}{2}
+ o(1)
~=~
\min\{\gamma_E, \gamma_V\} \smfrac{\| v_i \|^2}{2}
+ o(1)
,
\end{align*}
where in the last step we used that $(\bar{v}_i, \tilde{v}_i)_{\ell^2}\to (v_\infty, 0)_{\ell^2} =0$.
This is a contradiction since
\[
\liminf_{i\to\infty}
\frac{\tilde{E}_{0,n_i,k_i}(v_i )}{\|v_i\|^2} \geq \smfrac12 \min\{\gamma_E, \gamma_V\}.
\]

\medskip \noindent {\it Case 2 ($\lim_{i\to\infty}\|v_i\| =\infty$).}
Similarly expanding $E_{0,k_i}(u^*_{n_i} + v_i)$ and using \eqref{eq:orig_ass:P_bound}, \eqref{eq:orig_ass:V_growth}, \eqref{eq:orig_ass:Vbc_growth1}, and \eqref{eq:orig_ass:Vbc_growth2} quickly leads to a contradiction:
\[
\frac{E_{0,k_i}(u^*_{n_i} + v_i)}{\|v_i\|^2}
-\frac{E_{0,k_i}(u^*_{n_i})}{\|v_i\|^2}
\geq
\frac{-\tilde{C}_\calP (1+\|u^*_{n_i}+v_i\|)}{\|v_i\|^2}
+ \frac{\gamma_V \|u^*_{n_i}+v_i\|^2}{2 \|v_i\|^2}
-\frac{E_{0,k_i}(u^*_{n_i})}{\|v_i\|^2}
\to
	\smfrac{\gamma_V}{2}
.
\]

\medskip \noindent {\it Case 3 ($\lim_{i\to\infty}\|v_i\|=0$).}
Due to a uniform bound on $\dddel \tilde{E}_{0,n_i,k_i}(v)$ near $v=0$ we have
\begin{align*}
\frac{2 \tilde{E}_{0,n_i,k_i}(v_i)}{\|v_i\|^2}
=~&
\|v_i\|^{-2} \<\ddel \tilde{E}_{0,n_i,k_i}(0) v_i, v_i\> + O(\|v_i\|)
\\=~&
\|v_i\|^{-2} \<\ddel E_{0,k_i}(u^*|_{\calB_{k_i}}) v_i, v_i\> + o(1)
\\\geq~&
\min\{\gamma_E,\gamma_V\} + o(1)
,
\end{align*}
where in the last step we used \eqref{eq:defect_stab}.
This is again a contradiction.
\end{proof}

{\bf We will omit tildes in the rest of this section.}
In particular, we choose $\eps$ in Lemma \ref{lem:ass} so that $\tilde{\gamma}_V$ and $\tilde{\gamma}_E$ are both positive.

\subsection{Expansions of Gaussian-like Integrals}\label{sec:expansion_formulation}

Here we formulate two lemmas stating some expansions of Gaussian-like integrals in power series.
Since these results are just variations of some standard results, we provide the proofs in Appendix \ref{sec:expansion_proofs}.

For $f:\bbR^n\to\bbR$ we define the norms $\|f\|_{C_{\gamma}} := \sup_{x\in\bbR^n} \big| f(x) e^{\frac{\gamma}{2} |x|^2}\big|
$ and $\|f\|_{L^2_\gamma}:= \| f e^{\frac{\gamma}{2} |\bullet|^2} \|_{L^2} $, and the associated spaces of continuous functions $C_{\gamma}$ and Sobolev spaces $L^2_{\gamma}$.

\begin{lemma}\label{lem:Gaussian_expansion}
Let $E_\theta(u)$, $E:\bbR_+\times\bbR^n \to \overline{\bbR}$ be defined by
\begin{align}\label{eq:Gaussian_expansion:E}
E_\theta(v) =~& \sum_{m=0}^{M} \theta^{m} E^{(m)}(v)
+ \theta^{1+M} E^{(1+M^+)}_\theta(v)
\end{align}
for some $M\in\bbN$.
Assume for all $0<\theta\leq\theta_0$ and $m=0,\ldots,M$,
\begin{align} \notag
&&&&&\hspace{-7em}
E^{(m)}, E^{(1+M^+)}_\theta \text{ are measurable with values in $\overline{\bbR}$}
,\\ \label{eq:Gaussian_expansion:Eell_norm}
&&&&&\hspace{-7em}
\big\|E^{(m)}\big\|_{\C^{2M-2m+3}(B_{2r})} \leq c
,\\ \label{eq:Gaussian_expansion:remainder_est}
&&&&&\hspace{-7em}
\big\|E_\theta^{(1+M^+)}\big\|_{\C(B_{2r})} \leq c
,\\ \label{eq:Gaussian_expansion:Ezero}
&&& E^{(0)}(0) = 0
,&&
E^{(0)}(v) \geq \gamma |v|^2/2
\quad\forall v\in\bbR^n
,&&\\ \label{eq:Gaussian_expansion:Erest_growth}
&&& E^{(m)}(v) \geq -c
\quad\forall v\in\bbR^n,
&&
E_\theta^{(1+M^+)}(v) \geq -c
\quad\forall v\in\bbR^n
,&&
\end{align}
for some $\gamma\geq 0$, $c>0$, $r>0$.
Denote $H := \ddel E^{(0)}(0)$ and $L := 2M$.
Then
\begin{itemize}
\item[(a)]
there exists a unique Taylor expansion, in terms of powers of $\sqrt{\theta}$,
\begin{equation}\label{eq:Gaussian_expansion_predef}
e^{\frac12 \<H v,v\>} e^{-\theta^{-1} E_\theta(\sqrt{\theta} v)}
=
\sum_{\ell=0}^{L-1} \theta^{\ell/2} Q^{(\ell)}(v) + \theta^{L/2} Q_\theta^{(L^+)}(v)
,
\end{equation}
where (i) each $Q^{(\ell)}$ is a polynomial of degree at most $\ell+2$ \emph{whose coefficients are linear combinations of $\del^k E^{(m)}(0)$ for $k+2m \leq L+2$}, (ii) $Q^{(0)}=e^{-E^{(1)}(0)}$, and (iii) $Q^{(\ell)}(-v) = (-1)^\ell Q^{(\ell)}(v)$.

\item[(b)]
Furthermore, the terms and the remainder of the following representation,
\begin{equation}\label{eq:Gaussian_expansion_def}
e^{-\theta^{-1} E_\theta(\sqrt{\theta} v)}
=
\sum_{\ell=0}^{L-1} \theta^{\ell/2} Q^{(\ell)}(v) e^{-\frac12\<H v,v\>} + \theta^{L/2} Q_\theta^{(L^+)}(v) e^{-\frac12\<H v,v\>}
,
\end{equation}
can be bounded as follows:
\begin{align}\label{eq:Gaussian_expansion:terms}
|Q^{(\ell)}(v) e^{-\frac12\<H v,v\>}|
\leq~&
C_{\eps} e^{-\frac{\gamma-\eps}{2} \|v\|^2}
~~~~\forall v\in\bbR^n,
\\ \label{eq:Gaussian_expansion:remainder}
|Q_\theta^{(L^+)}(v) e^{-\frac12\<H v,v\>}|
\leq~&
C_{\eps} e^{-\frac{\gamma-\eps}{2} \|v\|^2}
~~~~\forall v\in\bbR^n
~~~~\forall \theta\in(0,\theta_0]
,
\end{align}
for any $\eps>0$ where $C_{\eps}$ depends only on $\eps$, $M$, $c$ and $n$.
\end{itemize}
\end{lemma}

\begin{lemma}\label{lem:expansion_generic}
Let the $E$ functions, $M$, $L$, $c$, and $H := \ddel E^{(0)}$ be as in Lemma \ref{lem:Gaussian_expansion}.
Let $F_\theta(v)$, $F:\bbR^+\times\bbR^n \to \overline{\bbR}$ be defined by
\begin{align*}
F_\theta(v) =~& \sum_{\ell=0}^{L-1} \theta^{\ell/2} F^{(\ell)}(v)
+ \theta^{L/2} F_{\theta}^{(L^+)}(v),
\end{align*}
where for all $\theta\leq\theta_0$ and $\ell=0,1,\ldots,L-1$
\begin{align}
\notag
&&& \text{$F^{(\ell)}$ is measurable},
&&\|F^{(\ell)}\|_{L^2_{-\gamma'}(\bbR^n)} \leq c,
&&\\ \notag
&&& \text{$F_\theta^{(L^+)}$ is measurable},
&& \|F_\theta^{(L^+)}\|_{L^2_{-\gamma'}(\bbR^n)} \leq c
&&\\
&&& F_\theta(v) \geq 0
\quad\forall v\in\bbR^n
,
&&
F^{(0)}(v) \geq c^{-1}
\quad \forall v\in B_{c^{-1}}
,
&&
\label{eq:expansion_generic:first_term_ass}
\end{align}
where $\gamma'<\gamma$.

Consider a measure $\mu_{\theta,E,F}$ defined by
\begin{equation}\label{eq:expansion_generic:def}
\<\mu_{\theta,E,F}, A \> :=
\frac{%
	\int_{\bbR^n} A(\sqrt{\theta}v)
	F_\theta(v) e^{-\theta^{-1} E_\theta(\sqrt{\theta}v)} \dv
}{%
	\int_{\bbR^n}
	F_\theta(v) e^{-\theta^{-1} E_\theta(\sqrt{\theta}v)} \dv
}
.
\end{equation}

\begin{itemize}
\item[(a)]
Then there exist distributions $\mu_{E,F}^{(m)}$ $(m=0,\ldots,M-1)$ and $\mu_{\theta,E,F}^{(M+)}$ such that
\begin{align} \label{eq:expansion_generic}
\mu_{\theta,E,F}
=
\sum_{m=0}^{M-1}
\theta^{m} \mu_{E,F}^{(m)}
+ \theta^{M/2} \mu_{\theta,E,F}^{(M+)}
\qquad \forall \theta>0
,
\end{align}

\item[(b)] There exists $\theta_0$ and $C$, that may depend on $\gamma$, $\gamma'$, $M$, $c$, and $n$, such that
\[
\max_{0\leq m<M} \big\|\mu_{E,F}^{(m)}\big\|_{\C^{-L}_0}\leq C
\quad\text{and}\quad
\big\|\mu_{\theta,E,F}^{(M^+)}\big\|_{\C^{-L}_0}\leq C
\qquad \forall \theta\leq \theta_0
.
\]

\item[(c)]
Furthermore, each $\mu_{E,F}^{(m)}$ is a linear combination of the Dirac delta and its derivatives.
In particular,
\[
\<\mu_{E,F}^{(0)}, A\> = A(0)
\]
and additionally, in the case when $F_\theta(v) \equiv 1$,
\begin{align*}
\<\mu^{(1)}, A\> \equiv~&
\smfrac12 \ddel A(0) \!:\! H^{-1}
-
\del A(0) \cdot H^{-1} \big(
	\del E^{(1)}(0)+\smfrac12\dddel E^{(0)}(0) \!:\! H^{-1}
\big)
.
\end{align*}

\item[(d)] For all $m=0,1,\ldots,L-1$, $\mu_{E,F}^{(m)}$ depends continuously on $\del^\ell E^{(k)}(0)$ ($\ell+2 k\leq 2m+2$) and $F^{(\ell)} \in L^2_{-\gamma'}$ ($\ell \leq m$).
\end{itemize}
\end{lemma}

\subsection{Main Line of Arguments}

We start with the following result.
\begin{theorem}\label{th:some_expansion}
Parts (a) and (d, $N<\infty$) of
Theorem \ref{th:uniform_expansion} hold.
Moreover, part (c, $N<\infty$) also holds if the estimates \eqref{eq:uniform_expansion:bounds} are relaxed as follows:
\[
\max_{0\leq\ell<\lfloor (L-1)/2 \rfloor} \big\|\mu_{N}^{(m)}\big\|_{\C^{-L}_0}\leq C'(N)
\quad\text{and}\quad
\big\|\mu_{\theta,N}^{(M^+)}\big\|_{\C^{-L}_0}\leq C'(N)
\qquad \forall\theta\leq \theta_0
,
\]
where the bound $C'(N)$ may depend on $N$ (but is independent of $\theta$).
\end{theorem}
\begin{proof}
This is an immediate consequence of Lemma \ref{lem:expansion_generic} with $F(\theta, v) \equiv 1$ upon checking that its assumptions follow from the assumptions on $V$ and $V^\bc$.
\end{proof}


Thus, for the case $N<\infty$, it remains to show that the bound $C'(N)$ can be chosen independently of $N$.
To that end we introduce the following proposition, whose formulation and proof are totally analogous to \cite{BlancLeBrisLegollEtAl2010}.

\begin{proposition}\label{prop:expansion_reduction}
For $A\in \C^{2M}(\calU_{K})$ there holds
\begin{equation}\label{eq:expansion_form}
\<\mu_{\theta,N}, A\>
=
\frac{%
	\int_{\calU_{K}}
	A(u^*_N + \sqrt{\theta} v)
	e^{-\theta^{-1} E_{N,K}(\sqrt{\theta} v)}
	\hat{\varphi}_{\theta,N,K}(v_{K})
	\hat{\varphi}_{\theta,N,K}(v_{-K})
	\dv
}{%
	\hphantom{A(u^*_N + \sqrt{\theta} v)}
	\int_{\calU_{K}}
	e^{-\theta^{-1} E_{N,K}(\sqrt{\theta} v)}
	\hat{\varphi}_{\theta,N,K}(v_{K})
	\hat{\varphi}_{\theta,N,K}(v_{-K})
	\dv
}
\end{equation}
where
\begin{align}\notag
	E_{N,K}(v)
	:=~&
	\calP(v) +
	\sum_{\xi=-K+\half}^{K-\half} V_{N,\xi}(v_{\xi-\half}, v_{\xi+\half})
,\\
\label{eq:hatphi-def}
	\hat{\varphi}_{\theta,N,K}(x)
	:=~&
	\frac{\varphi_{\theta,N,K}(x)}{\|\varphi_{\theta,N,K}\|_{L^2}}
,\\
\label{eq:phi-def}
	\varphi_{\theta,N,K}(x)
	:=~&
	\oprod_{\xi=K+\half}^{N-\half}
			P_{\theta,N,\xi}[p^\bc_\theta](x)
,\\
\label{eq:Ptheta-def}
	P_{\theta,N,\xi}[\varphi](x)
	:=~&
	\int_{-\infty}^{\infty}
		p_{\theta,N,\xi}(x,y) \varphi(y)
	\dy
	\quad \forall\theta>0
,
\\ \notag
	p_{\theta,N,\xi}(x,y)
	:=~&
	e^{-\theta^{-1} V_{N,\xi} (\sqrt{\theta}x,\sqrt{\theta}y)}
	\quad \forall\theta>0
,\\ \notag
	p_{\theta,N}^\bc(x)
	:=~&
	e^{-\theta^{-1} V^\bc_{\theta,N} (\sqrt{\theta}x)}
	\quad \forall\theta>0
,
\end{align}
and where $\smalloprod$ stands for the composition of multiple operators; provided that $\|\varphi_{\theta,N,K}\|_{L^2}<\infty$.
\end{proposition}
\begin{proof}
We have
\begin{align*}
\<\mu_{\theta,N}, A\>
=~&
\frac{%
	\int_{-\infty}^{\infty} A(u^*_N + v_0) e^{-\theta^{-1} E_{\theta,N}(v)} \dv_0
	\int_{-\infty}^{\infty} \dv_1\ldots\int_{-\infty}^{\infty} \dv_{N}
	\int_{-\infty}^{\infty} \dv_{-1}\ldots\int_{-\infty}^{\infty} \dv_{-N}
}{%
	\int_{-\infty}^{\infty} e^{-\theta^{-1} E_{\theta,N}(v)} \dv_0
	\int_{-\infty}^{\infty} \dv_1\ldots\int_{-\infty}^{\infty} \dv_{N}
	\int_{-\infty}^{\infty} \dv_{-1}\ldots\int_{-\infty}^{\infty} \dv_{-N}
},
\end{align*}
where we can write
\begin{align*}
e^{-\theta^{-1} E_{\theta,N}(\sqrt{\theta}v)}
=~&
p^\bc_{\theta,N}(v_{-N})
p^\bc_{\theta,N}(v_{N})
e^{-\theta^{-1} E_{N,N}(\sqrt{\theta}v)}
\\=~&
p^\bc_{\theta,N}(v_{-N})
p^\bc_{\theta,N}(v_{N})
e^{-\theta^{-1} E_{N, K}(\sqrt{\theta}v)}
\\ &\cdot
\prod_{\xi=K+\half}^{N-\half}
p_{ \theta,N, \xi}(v_{\xi-\half},v_{\xi+\half})
\prod_{\xi=-N+\half}^{-K-\half}
p_{ \theta,N, \xi}(v_{\xi-\half},v_{\xi+\half})
.
\end{align*}

Changing the variables $v_k\mapsto \sqrt{\theta} v_k$ in the integrals and using the definition of $P_\theta$ yields
\begin{align*}
&\<\mu_{\theta,N}, A\>
\\=&
\frac{%
	\int_{\calU_{K}}
	A(u^*_N + \sqrt{\theta} v)
	e^{-\theta^{-1} E_{N,K}(\sqrt{\theta} v)}
	\oprod_{\xi=K+\half}^{N-\half}
		P_{\theta,N,\xi}[p^\bc_{\theta,N}](v_{K})
	\oprod_{\xi=-N+\half}^{-K-\half}
		P_{\theta,N,\xi}[p^\bc_{\theta,N}](-v_{K})
	\dv
}{%
	\hphantom{A(u^*_N + \sqrt{\theta} v)}
	\int_{\calU_{K}}
	e_{N,K}^{-\theta^{-1} E_{N,K}(\sqrt{\theta} v)}
	\oprod_{\xi=K+\half}^{N-\half}
		P_{\theta,N,\xi}[p^\bc_{\theta,N}](v_{K})
	\oprod_{\xi=-N+\half}^{-K-\half}
		P_{\theta,N,\xi}[p^\bc_{\theta,N}](-v_{K})
	\dv
}
.
\end{align*}
It remains to note the symmetry
\[
\varphi_{\theta,N,K} = ~~
\oprod_{\xi=K+\half}^{N-\half}
	P_{\theta,N,\xi}[p^\bc_{\theta,N}]
=
\oprod_{\xi=-N+\half}^{-K-\half}
	P_{\theta,N,\xi}[p^\bc_{\theta,N}]
\]
and that the rescaling of $\varphi_{\theta,N,K}$ by its norm does not change the expression \eqref{eq:expansion_form}.
\end{proof}

Next, we formulate a lemma that adsorbs all the technicalities of proving Theorem \ref{th:uniform_expansion}.
Below by default we let $L := 2M$.

\begin{lemma}\label{lem:uniform_est:key}
Under assumptions of Theorem \ref{th:uniform_expansion}, there exist $K_0$, $\hat{\varphi}_{N,K}^{(\ell)}$ ($\ell=0,1,\ldots,L-1$), and $\hat{\varphi}_{\theta,N,K}^{(L^+)}$ such that
\begin{align} \label{eq:phi_n_expansion}
&&\hat{\varphi}_{\theta,N,K} =~&
	\sum_{\ell=0}^{L-1} \theta^{\ell/2} \hat{\varphi}_{N,K}^{(\ell)}
	+ \theta^{L/2} \hat{\varphi}_{\theta,N,K}^{(L^+)}
,
\\ \label{eq:uniform_est:nonnegativity}
&&\hat{\varphi}_{\theta,N,K}(x) \geq~& 0
\qquad \forall x\in\bbR
\\ \label{eq:uniform_est:first_term_positivity}
&&\hat{\varphi}_{N,K}^{(0)}(x) \geq~& c^{-1}
\qquad \forall x:|x| \leq c^{-1}
\\ \label{eq:uniform_est:norm_est}
&&
\big\|\hat{\varphi}_{N,K}^{(\ell)}\big\| \leq~& c
\quad \forall\ell=0,1,\ldots,L-1
&
\big\|\hat{\varphi}_{\theta,N,K}^{(L^+)}\big\| \leq~& c
\quad \forall \theta\leq \theta_0
&&
\end{align}
with some $\theta_0>0$, $c>0$, for all $K$ and $N$ such that $K_0 \leq K \leq N$.
\end{lemma}

The entire Section \ref{sec:proof_of_main_lemma} is dedicated to the proof of this lemma which is summarized in the end of Section \ref{sec:linear_est}.

Now we are ready to prove the remaining statements of Theorem \ref{th:uniform_expansion} for the case $N<\infty$.
\begin{proof}[Proof of Theorem \ref{th:uniform_expansion}, case $N<\infty$.]
In view of Theorem \ref{th:some_expansion}, we only need to ensure that there exists a respective $N$-independent bound for \eqref{eq:uniform_expansion:bounds}.

Let $A\in \C^{L}(\calU_{K})$ for some $K$ and let $K_0$ be given by Lemma \ref{lem:uniform_est:key}.
Without loss of generality we can consider $K\geq K_0$ (since otherwise $\C^{L}(\calU_{K})\subset \C^{L}(\calU_{K_0})$).
We can then always choose $C \geq \max_{n \leq K}C'(n)$, where $C'(N)$ is the bound provided by Theorem \ref{th:some_expansion}.
Therefore we only need to consider the case $N> K$.

We are thus in case when we can apply Lemma \ref{lem:uniform_est:key}.
Its statements verify the conditions of Lemma \ref{lem:expansion_generic} applied to \eqref{eq:expansion_form} with $F_\theta(v) = \varphi_{\theta,N,K}(v_{-K}) \varphi_{\theta,N,K}(v_{K})$.
Note that $\|F_\theta\|_{L^2_{\gamma'}}$ is bounded by $\|\varphi_{\theta,N,K}\|_{L^2}$ for any $\gamma'<0$, and the quadratic growth of $E_{N,K}$ is guaranteed by Lemma \ref{lem:Enk_convex} applied with $V^\bc_{\theta,N} \equiv 0$.
The bound on $\mu_{N}^{(m)}$ and $\mu_{\theta,N}^{(M^+)}$ given by Lemma \ref{lem:expansion_generic} we denote by $C''(K)$.

We thus choose an $N$-independent bound for \eqref{eq:uniform_expansion:bounds}
\[
C(K) := \max\Big\{ \max_{n\leq K}C'(n), C''(K)\Big\}
.
\]
\end{proof}

\subsection{Proof of Lemma \ref{lem:uniform_est:key}}\label{sec:proof_of_main_lemma}

\subsubsection{Leading Order Term}\label{sec:linear_est}

We start with studying the leading order term in the expansion \eqref{eq:phi_n_expansion}, which can be obtained by formally letting $\theta\to 0$.

To that end, define $a_{N,\xi}=(V_{N,\xi})_{xx}(0,0)$, $b_{N,\xi}=(V_{N,\xi})_{xy}(0,0)$, $c_{N,\xi}=(V_{N,\xi})_{yy}(0,0)$.
Further define
\begin{align*}
p_{0,N,\xi}(x,y)
:=~&
\lim_{\theta\to 0} p_{\theta,N,\xi}(x,y)
= e^{-a_{N,\xi} x^2/2-b_{N,\xi} x y-c_{N,\xi} y^2/2}
,
\\
p^\bc_{0,N}(x)
:=~&
\lim_{\theta\to 0} p^\bc_{\theta,N}(x)
= e^{-(V_N^{\bc,(0)})_{xx}(0) x^2/2}
,
\end{align*}
and hence define $P_{0,N,\xi}$ using \eqref{eq:Ptheta-def} with $\theta=0$.

Our next goal is to get some insight on $P_{0,N,\xi}$.
This will help us in estimating $P_{\theta,N,\xi}[\varphi]$ for small $\theta$.

We start with the following proposition, which can be proved by a standard calculation.
\begin{proposition}\label{prop:P0_explicit}
For
$
\psi_{\sigma}(x) := \big(\smfrac{\sigma}{\pi}\big)^{\frac14} e^{-\frac{\sigma x^2}{2}}
$
there holds $\|\psi_{\sigma}\|=1$ and
\begin{equation}\label{eq:P0_explicit}
P_{0,N,\xi}[\psi_\sigma](x) =
\sqrt{\smfrac{2\pi}{c_{N,\xi}+\sigma}}
\, \big(\smfrac{\sigma}{\sigma'}\big)^{\frac14} \,
\psi_{\sigma'}
\end{equation}
with $\sigma'=a_{N,\xi}-\frac{b_{N,\xi}^2}{(c_{N,\xi}+\sigma)}$, provided that $c_{N,\xi}+\sigma>0$.
\qed
\end{proposition}

Proposition \ref{prop:P0_explicit} suggests that in the limit $\theta\to0$, $\varphi_{\theta,N,K}$ is a multiple of $\psi_{\sigma_{N,K}}$ for some sequence of $\sigma_{N,K}$.
This sequence is defined in the next lemma.

\begin{lemma}\label{lem:P0_recurrent}
Define
\begin{align*}
\sigma_{N,N} :=~& (V_N^{\bc,(0)})_{xx}(0),
\\
\sigma_{N,n-1} :=~& a_{N,n-\half}-\frac{b_{N,n-\half}^2}{c_{N,n-\half}+\sigma_{N,n}}
\qquad \forall n=N,N-1,\ldots,1
.
\end{align*}
Then for any $\delta>0$ one can find $K_0=K_0(\delta)$ independently of $N$ such that
\begin{align} \label{eq:sigma_est}
& 2 \geq a_{N,n-\half} \geq \sigma_{N,n} \geq
\begin{cases}
\gamma_\bc & n=N \\
\kappa^2/2 & n<N
\end{cases}
\qquad\forall n\geq K_0,
\\
\label{eq:sigma_minus_kappa}
& |\sigma_{N,n} - \kappa| \leq \delta
\qquad \forall n : K_0 \leq n \leq N-K_0
.
\end{align}
\end{lemma}
\begin{proof} 
We have that $a_{N,\xi} \to 1$, $b_{N,\xi}^2 \to 1-\kappa^2$, and $c_{N,\xi} \to 1-\delta_{N-\xi-\half} \gamma_\bc$ as $\xi,N\to\infty$, $\xi<N-\frac12$, due to the assumption \eqref{eq:ass:V_asymptote}.
Hence we assume that for any $\eps>0$ we can choose $\tilde{K}_0=\tilde{K}_0(\eps)$ sufficiently large such that $|a_{N,\xi}-1|\leq \eps$, $|b_{N,\xi}^2-(1-\kappa^2)|\leq \eps$, and $|c_{N,\xi}-(1-\delta_{N-\xi-\half} \gamma_\bc)|\leq \eps$ for $\xi\geq \tilde{K}_0(\eps)$.

{\it Proof of \eqref{eq:sigma_est}.}
First, we note that $\sigma_{N,N}\geq \gamma_\bc$ due to \eqref{eq:ass:Vbc_growth}.
Next, for $\xi\geq \tilde{K}_0(\eps)$, we have that from definition, if $\sigma_{N,n} \geq 0$ then
\[
\sigma_{N,n-1} \geq 1-\eps - \frac{1-\kappa^2+\eps}{(1-\eps)}
= \kappa^2 + O(\eps) \geq \kappa^2/2,
\]
provided that $\eps$ is chosen small enough.
Hence choosing $K_0=\tilde{K}_0(\eps)$ proves the lower bound in \eqref{eq:sigma_est} and the upper bound on $\sigma_{N,n}$ is obvious.

{\it Proof of \eqref{eq:sigma_minus_kappa}.}
Fix $\delta>0$.
Then one can calculate for $\xi > \tilde{K}_0(\eps)$
\begin{align*}
|\sigma_{N,\xi-\half}-\kappa|
=~& \bigg|
	a_{N,\xi}-\frac{b_{N,\xi}^2}{c_{N,\xi}+\sigma_{N,\xi+\half}}-\kappa
\bigg|
~=~
\bigg|1-\kappa-\frac{1-\kappa^2}{1+\sigma_{N,\xi+\half}}\bigg| + O(\eps)
\\ = ~&
\frac{1-\kappa}{1+\sigma_{\xi+\half}} |\sigma_{N,\xi+\half}-\kappa| + O(\eps)
~\leq~
(1-\kappa)|\sigma_{N,\xi+\half}-\kappa| + O(\eps),
\end{align*}
where $O(\eps)$ indicates a bound which is independent of $N$ or $\xi$.
It is then not hard to show that hence
$
|\sigma_{N,n}-\kappa|
\leq
(1-\kappa)^{N-n} |\sigma_{N,N}-\kappa| + O(\eps)
$---this is proved in Lemma \ref{lem:recursive_bound} below.
Choosing $K_0$ large enough, so that
\[
(1-\kappa)^{N-n} |\sigma_{N,N}-\kappa| \leq (1-\kappa)^{K_0} |\sigma_{N,N}-\kappa| \leq \delta/2
\]
for $n\leq N-K_0$ and choosing $\eps$ small enough yields \eqref{eq:sigma_minus_kappa}.
\end{proof}

We have used the following lemma.
\begin{lemma}\label{lem:recursive_bound}
Let a sequence $x_{n}\in\bbR$, $n=0,1,\ldots$, be such that $x_{n+1} \leq \mu x_{n} + c$ ($n=1,2,\ldots$) with $0\leq \mu<1$.
Then
\[
x_{n} \leq x_0 \mu^n + \frac{c (1-\mu^n)}{1-\mu}
\leq \max\Big\{x_0, \frac{c}{1-\mu}\Big\}
\qquad \forall n\geq 0.
\]
\end{lemma}
\begin{proof}
The first estimate is trivially proved by induction. The second estimate follows from the fact that $x_0 \mu^n + \frac{c (1-\mu^n)}{1-\mu}$ is a convex combination of $x_0$ and $\frac{c}{1-\mu}$.
\end{proof}

We are now ready to formulate a more detailed version of Lemma \ref{lem:uniform_est:key}.

\begin{lemma}\label{lem:interation_boundness}
Let $\hat{\varphi}_{\theta,N,K}$ be defined by \eqref{eq:hatphi-def}.
There exists $\theta_0>0$ independent of $N$, such that for all $N$, $n$, $\ell=0,1,\ldots,L-1$, and $\theta\leq\theta_0$, the expansion \eqref{eq:hatphi-def} holds with
\begin{gather}
\label{eq:phi_n_expansion_term_start}
\hat{\varphi}_{N,n}^{(0)}
=
\psi_{\sigma_{N,n}}
\\ \label{eq:phi_n_expansion_term_mid}
\|\hat{\varphi}_{N,n}^{(\ell)}\|
\,\leq~
C_{\ell}
\qquad\qquad
\|\hat{\varphi}_{\theta,N,n}^{(L^+)}\|
\,\leq~
C_{L^+}
\end{gather}
for any $n=K_0,\ldots,N$, where $C_\ell$ and $C_{L^+}$ are some constants independent of $N$, $n$, or $\theta$.
\end{lemma}

Lemma \ref{lem:interation_boundness} will be proved in Section \ref{sec:proof_of_interation_boundness}.
We now show that Lemma \ref{lem:uniform_est:key} trivially follows from Lemma \ref{lem:interation_boundness}:
\begin{proof}[Proof of Lemma \ref{lem:uniform_est:key}]
The property \eqref{eq:uniform_est:nonnegativity} follows from tracking the definition of $\varphi_{\theta,N,n}$ back to $V_{N,\xi}$ and $V_{N}^{\bc}$ and noting that $\varphi_{\theta,N,n}$ is an exponent of some function.
Then, \eqref{eq:uniform_est:norm_est} follows from \eqref{eq:phi_n_expansion_term_mid}.
Finally, \eqref{eq:uniform_est:first_term_positivity} follows from the definition of $\psi_\bullet$: indeed, for $|x| \leq c^{-1}$
\[
\hat{\varphi}_{N,n}^{(0)}(x)
=
\psi_{\sigma_{N,n}}(x)
\geq
\big(\smfrac{\sigma_{N,n}}{\pi}\big)^{\frac14}
e^{-\sigma_{N,n} \frac{c^{-2}}{2}},
\]
which is bounded uniformly in $N$ and $n$ thanks to the bounds on $\sigma_{N,n}$, \eqref{eq:sigma_est}. This concludes the proof of \eqref{eq:uniform_est:first_term_positivity}.
\end{proof}

\subsubsection{Several Auxiliary Results}

\begin{lemma} \label{lem:operator_norm_bound}
Let a measurable function $q(x,y)$ be such that
$
|q(x,y)|
\leq
c e^{-\gamma (x^2+y^2)/2}
$
for some $\gamma_1>0$ and define
\[
Q[\varphi](x)
:=
\int_{-\infty}^{\infty}
\int_{-\infty}^{\infty}
q(x,y) \varphi(y)
\dx \dy.
\]
Then $\|Q\|_{L^2} \leq C$, where $C=C(c,\gamma)$.

Furthermore, for a family of functions $q_u(x,y)$ parametrized by $u\in U$ with $U$ being an open set of some Banach space, such that $q\in \C^m(U; \C_{\gamma}(\bbR^2))$ for some $m\in\bbZ_+$, the respective operators $Q_u$ are such that $Q\in \C^m(U; B(L^2, L^2))$, where $B(\bullet, \bullet)$ stands for the space of bounded linear operators.
\end{lemma}
\begin{proof}
We have
\begin{align*}
\|Q[\varphi]\|^2
\leq~&
	\int_{-\infty}^{\infty}
	\int_{-\infty}^{\infty}
	|q(x,y)|^2 |\varphi(y)|^2
	\dx \dy
\leq
	\|\varphi\|^2
	\int_{-\infty}^{\infty}
	\int_{-\infty}^{\infty}
	|q(x,y)|^2
	\dx \dy
\leq
	C \|\varphi\|^2,
\end{align*}
which proves the first part. The second part follows upon a standard application of the Lebesgue's dominated convergence theorem.
\end{proof}

\subsubsection{Proof of Lemma \ref{lem:interation_boundness}}\label{sec:proof_of_interation_boundness}

The proof of Lemma \ref{lem:interation_boundness} is conducted through induction over $n$ and is contained in the lemmas below.

Suppose we have an expansion of $\varphi_{\theta,N,n}$ and we need to expand \eqref{eq:hatphi-def}.
To do that, we need to invert expansions in $\theta$, which is done in the following lemma.

\begin{lemma}\label{lem:inverse_expansion}
Let $p\in\bbR$, $X$ be a Hilbert space, and $f:\bbR^+\to X$ have an expansion
\begin{equation}\label{eq:inverse_expansion_original}
f(\theta) = \sum_{\ell=0}^{L} \theta^{\ell/2} f^{(\ell)}
\end{equation}
for some $L\in\bbN$, and assume that $\|f(\theta)\| \ne 0$ for $\theta\in[0,\theta_0]$.
\begin{itemize}
\item[(a)]
$F(\theta) := \|f(\theta)\|^p$ has the following finite Taylor expansion for $\theta\in[0,\theta_0]$
\begin{equation}\label{eq:inverse_expansion}
F(\theta) = \sum_{\ell=0}^{L} \theta^{\ell/2} F^{(\ell)}[f^{(0)},\ldots,f^{(\ell)}] + \theta^{\frac12+L/2} F^{(1+L^+)}[f^{(0)},\ldots,f^{(L)};\sqrt{\theta}]
,\quad \forall \theta \leq \theta_0,
\end{equation}
where
\begin{align}
	F^{(0)}[f^{(0)}]
=~& \label{eq:inverse_expansion:repr0}
	\big\|f^{(0)}\big\|^p,
\\
	F^{(\ell)}[f^{(0)},\ldots,f^{(\ell)}]
=~& \label{eq:inverse_expansion:repr}
	p \big\|f^{(0)}\big\|^{p-2} \big(f^{(0)},f^{(\ell)}\big)_X
	+\tilde{F}^{(\ell)}[f^{(0)},\ldots,f^{(\ell-1)}],
	\quad
	\forall \ell=1,\ldots,L.
\end{align}

\item[(b)]
For any $c>0$ there exists $\theta_1(c)>0$ such that all $F^{(\ell)}$, $\tilde{F}^{(\ell)}$, and $F^{(1+L^+)}$ are of the class $\C^m$ for any $m\in\bbZ_+$ on the set $f^{(0)} \geq c^{-1}$, $\max_{1\leq \ell\leq L} |f^{(\ell)}| \leq c$ and $|\theta|\leq \theta_1(c)$.
\end{itemize}
\end{lemma}

The proof of this lemma is quite technical, but the statement does not appear surprising.
We therefore give the proof in the Appendix \ref{sec:composition}.

The next lemma constitutes the basis of induction.
\begin{lemma}\label{lem:interation_boundness:induction_basis}
Under assumptions of Lemma \ref{lem:interation_boundness}, the expansion \eqref{eq:phi_n_expansion} and the identities
\eqref{eq:phi_n_expansion_term_start}--\eqref{eq:phi_n_expansion_term_mid} hold for $n=N$ and $\theta\in(0,\theta_0]$ for some $\theta_0>0$.
\end{lemma}
\begin{proof}
We apply Lemma \ref{lem:Gaussian_expansion} to
\[
\varphi_{\theta,N,N} = p_{\theta,N}^\bc
=
\sum_{\ell=0}^{L-1}
\theta^{\ell/2} \varphi_{N,N}^{(\ell)}
+ \theta^{L/2} \varphi_{\theta,N,N}^{(L^+)}
,
\]
and obtain the following bounds similar to \eqref{eq:phi_n_expansion_term_mid} but on $\varphi_{N,N}^{(\ell)}$ and $\varphi_{\theta,N,N}^{(L^+)}$:
\[
\|\varphi_{N,N}^{(\ell)}\|
\leq
C_{\ell},
\qquad\qquad
\|\varphi_{\theta,N,N}^{(L^+)}\|
\leq
C_{L^+}
.
\]
Then, thanks to the following lower bound,
\[
\|\varphi_{0,N,N}\|
= \big\|e^{-\frac{\sigma_{N,N}}{2} \bullet^2}\big\|
= \big(\smfrac{\pi}{\sigma_{N,N}}\big)^{1/2}
\geq
\big(\smfrac{\pi}{\gamma_V/2}\big)^{1/2}
> 0
\]
we can apply Lemma \ref{lem:inverse_expansion} and obtain the expansion
\begin{align}
\|\varphi_{\theta,N,N}\|^{-1}
=~&
\sum_{\ell=0}^{L-1} \theta^{\ell/2} F^{(\ell)}[\varphi_{N,N}^{(0)},\ldots,\varphi_{N,N}^{(\ell)}] + \theta^{L/2} F^{(L^+)}[\varphi_{N,N}^{(0)},\ldots,\varphi_{\theta,N,N}^{(L)};\sqrt{\theta}]
\label{eq:induction_basis:inv_exp}
\end{align}
valid for $\theta\in(0,\theta_0]$ for some $\theta_0>0$ (here $F^{(L^+)} = F^{(L)} + \sqrt{\theta} F^{(1+L^+)}$).
Moreover, the continuity of functions $F^{(\bullet)}$ (recall that Lemma \ref{lem:inverse_expansion}(b) yields $F^{(\bullet)}\in \C^{m}$) translates into the boundedness of the respective terms of the expansion \eqref{eq:induction_basis:inv_exp}.
After multiplying this expansion by the expansion of $\varphi_{\theta,N,N}$, we finally obtain the expansion \eqref{eq:phi_n_expansion} and the bounds \eqref{eq:phi_n_expansion_term_mid}.
Finally, the identity \eqref{eq:phi_n_expansion_term_start} follows from multiplying the leading order terms in $\varphi_{\theta,N,N}$ and in $\|\varphi_{\theta,N,N}\|^{-1}$.
\end{proof}

In the induction step, we will use the following representation of $\hat{\varphi}_{\theta,N,n-1}$.

\begin{proposition}
Let $\tilde{\varphi}_{\theta,N,n-1} := P_{\theta,N,n-\half}[\hat{\varphi}_{\theta,N,n}]$.
Then $\hat{\varphi}_{\theta,N,n-1}$ can be represented as
\begin{equation}\label{eq:hatphi_representation}
\hat{\varphi}_{\theta,N,n-1}
=
\frac{\tilde{\varphi}_{\theta,N,n-1}}{\|\tilde{\varphi}_{\theta,N,n-1}\|}.
\end{equation}
\end{proposition}
\begin{proof}
The statement is obvious upon noting that $\varphi_{\theta,N,n-1}$, $\hat{\varphi}_{\theta,N,n-1}$ and $\tilde{\varphi}_{\theta,N,n-1}$ are all equal up to a multiplicative constant, which for $\hat{\varphi}_{\theta,N,n-1}$ is uniquely determined by the identity $\|\hat{\varphi}_{\theta,N,n-1}\|=1$.
\end{proof}

It remains to perform an induction step, that is, to prove the statements \eqref{eq:phi_n_expansion}, \eqref{eq:phi_n_expansion_term_start}--\eqref{eq:phi_n_expansion_term_mid} for $n-1$ assuming these statements hold for $n$. This is done in the following lemmas.

\begin{lemma}\label{lem:varphi_tilde_expansion}
If \eqref{eq:phi_n_expansion}, \eqref{eq:phi_n_expansion_term_start}--\eqref{eq:phi_n_expansion_term_mid} hold for $\xi+\smfrac12$ then there holds an expansion
\begin{equation}\label{eq:varphi_tilde_expansion_1}
\tilde{\varphi}_{\theta,N,\xi-\half}
=
\sum_{\ell=0}^{L-1}
\theta^{\ell/2} \tilde{\varphi}_{N,\xi-\half}^{(\ell)}
+ \theta^{L/2} \tilde{\varphi}_{\theta,N,\xi-\half}^{(L^+)}
,
\end{equation}
where
\begin{align}\label{eq:varphi_tilde_expansion_2}
\tilde{\varphi}_{N,\xi-\half}^{(\ell)}
=~&
P_{0,N,\xi} [\hat{\varphi}_{N,\xi+\half}^{(\ell)}] +
M_{N,\xi}^{(\ell)}[\hat{\varphi}_{N,\xi+\half}^{(0)},\ldots,\hat{\varphi}_{N,\xi+\half}^{(\ell-1)}]
\\ \notag
\tilde{\varphi}_{\theta,N,\xi-\half}^{(L^+)}
=~&
P_{0,N,\xi} [\hat{\varphi}_{\theta,N,\xi+\half}^{(L^+)}] +
M_{\theta,N,\xi}^{(L^+)}[\hat{\varphi}_{N,\xi+\half}^{(0)},\ldots,\hat{\varphi}_{N,\xi+\half}^{(L-1)}]
\\~&\hphantom{P_{0,N,\xi} [\hat{\varphi}_{\theta,N,\xi+\half}^{(L^+)}]}
+ \theta^{1/2} M_{\theta,N,\xi}^{(L^{++})}[\hat{\varphi}_{N,\xi+\half}^{(0)},\ldots,\hat{\varphi}_{N,\xi+\half}^{(L-1)},\hat{\varphi}_{\theta,N,\xi+\half}^{(L^+)}],
\label{eq:varphi_tilde_expansion_3}
\end{align}
where $M_{N,\xi}^{(\ell)}$, $M_{\theta,N,\xi}^{(L^+)}$, and $M_{\theta,N,\xi}^{(L^{++})}$ are some linear operators bounded uniformly in $N$ and $\theta$.
\end{lemma}
\begin{proof}
Thanks to the assumptions \eqref{eq:ass:V_bound}, \eqref{eq:ass:V_zero_order} and \eqref{eq:ass:V_growth}, and Lemma \ref{lem:Gaussian_expansion}, we have the following expansion
\begin{align*}
p_{\theta,N,\xi}(x,y)/p_{0,N,\xi}(x,y)
=~&
	e^{-\theta^{-1} V_{N,\xi}(\sqrt{\theta}x,\sqrt{\theta}y)}
	e^{(\frac12 a_{N,\xi} x^2 + b_{N,\xi} x y + \frac12 c_{N,\xi} y^2)}
\\ =~&
	\sum_{\ell=0}^{L-1} \theta^{\ell/2} q_{N,\xi}^{(\ell)}(x,y)
	+ \theta^{L/2} q_{\theta,N,\xi}^{(L^+)}(x,y)
\end{align*}
where $q_{N,\xi}^{(0)}(x,y) = 1$, $q_{N,\xi}^{(\ell)}(x,y)$ is a polynomial of $x$ and $y$ and
\[
\big|
	q_{N,\xi}^{(\ell)}(x,y) p_{0,N,\xi}(x,y)
\big|
\leq C e^{-\frac{\gamma'}{2} (x^2+y^2)}
,
\qquad
\big|
	q_{\theta,N,\xi}^{(L^+)}(x,y) p_{0,N,\xi}(x,y)
\big|
\leq C e^{-\frac{\gamma'}{2} (x^2+y^2)}
,
\]
where $\gamma':=\gamma_V/2$ and $C$ depends only on $L$, $\gamma_V$, and $c_V$.

Hence we have
\[
	P_{\theta,N,\xi}
=
	\sum_{\ell=0}^{L-1} \theta^{\ell/2} P_{N,\xi}^{(\ell)}
	+
	\theta^{L/2} P_{\theta,N,\xi}^{(L^+)},
\]
where
\begin{align*}
P_{N,\xi}^{(\ell)}[\varphi](x)
:=~&
\int_\bbR q_{N,\xi}^{(\ell)}(x,y) p_{0,N,\xi}(x,y) \varphi(y) \dy
,\qquad\text{and}
\\
P_{\theta,N,\xi}^{(L^+)}[\varphi](x)
:=~&
\int_\bbR q_{\theta,N,\xi}^{(L^+)}(x,y) p_{0,N,\xi}(x,y) \varphi(y) \dy.
\end{align*}
Notice that the first term in the operator sum is $P_{N,\xi}^{(0)} = P_{0,N,\xi}$.
Finally, according to Lemma \ref{lem:operator_norm_bound}, all operators have a uniform bound in $N$ and $\theta$.

Now, assume that the representation \eqref{eq:phi_n_expansion}, \eqref{eq:phi_n_expansion_term_start}--\eqref{eq:phi_n_expansion_term_mid} holds for $n=\xi+\frac12$ and hence expand
\begin{align*} \notag
\tilde{\varphi}_{\theta,N,\xi-\half}
=~&
P_{\theta,N,\xi}[\hat{\varphi}_{\theta,N,\xi+\half}]
\\ = ~&
\bigg(P_{0,N,\xi} + \sum_{\ell=1}^{L-1} \theta^{\ell/2} P_{N,\xi}^{(\ell)}
	+ \theta^{L/2} P_{\theta,N,\xi}^{(L^+)}
\bigg)
\bigg[\sum_{\ell=0}^{L-1}
\theta^{\ell/2} \hat{\varphi}_{N,\xi+\half}^{(\ell)}
+ \theta^{L/2} \hat{\varphi}_{\theta,N,\xi+\half}^{(L^+)}
\bigg]
\\ = ~&
\sum_{\ell=0}^{L-1} \theta^{\ell/2}
\bigg(P_{0,N,\xi} [\hat{\varphi}_{N,\xi+\half}^{(\ell)}] +
\underbrace{\sum_{m=0}^{\ell-1}
P_{N,\xi}^{(\ell-m)} [\hat{\varphi}_{N,\xi+\half}^{(m)}]}_{=: M_{N,\xi}^{(\ell)}[\hat{\varphi}_{N,\xi+\half}^{(0)},\ldots,\hat{\varphi}_{N,\xi+\half}^{(\ell-1)}]}
\bigg)
\\~&+
\theta^{L/2} \bigg(
	P_{0,N,\xi} [\hat{\varphi}_{\theta,N,\xi+\half}^{(L^+)}]
	+ \underbrace{\sum_{m=1}^{L-1}
	P_{N,\xi}^{(L-m)} [\hat{\varphi}_{N,\xi+\half}^{(m)}]
	+
	P_{\theta, N,\xi}^{(L^+)} [\hat{\varphi}_{N,\xi+\half}^{(0)}]}_{=:M_{\theta,N,\xi}^{(L^+)}[\hat{\varphi}_{N,\xi+\half}^{(0)},\ldots,\hat{\varphi}_{N,\xi+\half}^{(L-1)}]}
	\bigg)
\\ ~&+
\theta^{(L+1)/2} \bigg(~
	\sum_{\ell=1}^{L-1} \sum_{m=L+1-\ell}^{L-1}
	\theta^{(\ell+m-L-1)/2} P_{N,\xi}^{(\ell)} [\hat{\varphi}_{N,\xi+\half}^{(m)}]
\\ ~& \hphantom{+\theta^{(L+1)/2} \bigg(~} +
	\sum_{m=1}^{L-1}
	\theta^{(m-1)/2} P_{\theta,N,\xi}^{(L^+)} [\hat{\varphi}_{N,\xi+\half}^{(m)}]
\\ ~& \hphantom{+\theta^{(L+1)/2} \bigg(~} \underbrace{+
	\sum_{\ell=1}^{L-1}
	\theta^{(\ell-1)/2} P_{N,\xi}^{(\ell)} [\hat{\varphi}_{\theta,N,\xi+\half}^{(L^+)}]
+
	\theta^{(L-1)/2} P_{\theta,N,\xi}^{(L^+)} [\hat{\varphi}_{\theta,N,\xi+\half}^{(L^+)}]}_{=:M_{\theta,N,\xi}^{(L^{++})}[\hat{\varphi}_{N,\xi+\half}^{(0)},\ldots,\hat{\varphi}_{N,\xi+\half}^{(L-1)},\hat{\varphi}_{\theta,N,\xi+\half}^{(L^+)}]}
\bigg)
.
\end{align*}

The uniform boundedness of the $M$ operators follows from that of $P$ operators.

\end{proof}

Next, to apply Lemma \ref{lem:inverse_expansion} to the expansion of $\|\tilde{\varphi}_{\theta,N,\xi-\half}\|^{-1}$ in \eqref{eq:hatphi_representation}, we need to check the following condition on its denominator.
\begin{lemma}\label{lem:nv_xi_positivity}
The denominator of \eqref{eq:hatphi_representation} at $\theta=0$,
$
\nu_{N,n}
:=
\|\tilde{\varphi}_{N,n}^{(0)}\|
,
$
is uniformly positive with respect to $N$ and $n$.
\end{lemma}
\begin{proof}
Proposition \ref{prop:P0_explicit} immediately yields that $\nu_{N,n} = \sqrt{\smfrac{2\pi}{c_{N,n+\half}+\sigma_{N,n+1}}}
\, \big(\smfrac{\sigma_{N,n+1}}{\sigma_{N,n}}\big)^{\frac14}$, which is uniformly positive thanks to \eqref{eq:sigma_est}.
\end{proof}

\begin{lemma}\label{lem:elimination_operators}
If \eqref{eq:phi_n_expansion}, \eqref{eq:phi_n_expansion_term_start}--\eqref{eq:phi_n_expansion_term_mid} hold for $n=\xi+\half$ then \eqref{eq:phi_n_expansion} holds for $n=\xi-\half$ where
\begin{align}\label{eq:varphi_hat_expansion_1}
\hat{\varphi}_{N,\xi-\half}^{(0)}
=~&
\hat{P}_{0,N,\xi} [\hat{\varphi}_{N,\xi+\half}^{(0)}]
\\ \label{eq:varphi_hat_expansion_2}
\hat{\varphi}_{N,\xi-\half}^{(\ell)}
=~&
\hat{Q}_{N,\xi} [\hat{\varphi}_{N,\xi+\half}^{(\ell)}] +
\hat{M}_{N,\xi}^{(\ell)}[\hat{\varphi}_{N,\xi+\half}^{(0)},\ldots,\hat{\varphi}_{N,\xi+\half}^{(\ell-1)}]
\\ \notag
\hat{\varphi}_{\theta,N,\xi-\half}^{(L^+)}
=~&
\hat{Q}_{N,\xi} [\hat{\varphi}_{\theta,N,\xi+\half}^{(L^+)}] +
\hat{M}_{\theta,N,\xi}^{(L^+)}[\hat{\varphi}_{N,\xi+\half}^{(0)},\ldots,\hat{\varphi}_{N,\xi+\half}^{(L-1)}]
\\~& \hphantom{\smfrac1{\nu_{N,\xi-\half}} P_{0,N,\xi} [\hat{\varphi}_{\theta,N,\xi+\half}^{(L^+)}]}
+ \theta^{1/2} \hat{M}_{\theta,N,\xi}^{(L^{++})}[\hat{\varphi}_{N,\xi+\half}^{(0)},\ldots,\hat{\varphi}_{N,\xi+\half}^{(L-1)},\hat{\varphi}_{\theta,N,\xi+\half}^{(L^+)}],
\label{eq:varphi_hat_expansion_3}
\end{align}
where $\hat{P}_{0,N,\xi} = \smfrac1{\nu_{N,\xi-\half}} P_{0,N,\xi}$, $\hat{Q}_{N,\xi} = {\Pr}_{\hat{\varphi}^{(0)}_{N,\xi-\half}} \hat{P}_{0,N,\xi}$, $\Pr_\psi[\varphi] := \varphi - \psi\frac{(\psi,\varphi)_{L^2}}{(\psi,\psi)_{L^2}}$ is a projection operator parallel to $\psi$,
and $\hat{M}_{N,\xi}^{(\ell)}$, $\hat{M}_{\theta,N,\xi}^{(L^+)}$, and $\hat{M}_{\theta,N,\xi}^{(L^{++})}$ are some operators acting on $L^2$.
Furthermore, for any $c>0$ there exists $\theta_0=\theta_0(c)$ such that these operators are Lipschitz uniformly in $N$, $\xi$, and $\theta\in(0,\theta_0]$ on the set $\big\|\hat{\varphi}_{N,\xi+\half}^{(\ell)}\big\| \leq c$ ($\ell=0,\ldots,L-1$), $\big\|\hat{\varphi}_{\theta,N,\xi+\half}^{(L^+)}\big\| \leq c$.
\end{lemma}
\begin{proof}
Lemma \ref{lem:nv_xi_positivity} guarantees that Lemma \ref{lem:inverse_expansion} can be applied to get an expansion of $\|\tilde{\varphi}_{\theta,N,\xi-\half}\|^{-1}$:
\begin{align}\notag
\|\tilde{\varphi}_{\theta,N,\xi-\half}\|^{-1}
=~&
\nu_{N,\xi-\half}^{-1} + \sum_{\ell=1}^{L} \theta^{\ell/2} Z^{(\ell)}[\hat{\varphi}_{N,\xi+\half}^{(0)},\ldots,\hat{\varphi}_{N,\xi+\half}^{(\ell)}]
\\~& \label{eq:varphi_denominator_scaling}
+
\theta^{\frac12+L/2} Z_{\theta}^{(1+L^+)}[\hat{\varphi}_{N,\xi+\half}^{(0)},\ldots,\hat{\varphi}_{N,\xi+\half}^{(L-1)},\hat{\varphi}_{\theta,N,\xi+\half}^{(L^+)}],
\end{align}
where, for any $c>0$, $Z^{(\ell)}$ and $Z_{\theta}^{(L^+)}$, being combinations of the $F$ functions from Lemma \ref{lem:inverse_expansion} and the $M$ functions from Lemma \ref{lem:varphi_tilde_expansion}, are Lipschitz-continuous on the set $\|\hat{\varphi}_{N,\xi+\half}^{(\ell)}\| \leq c$ ($\ell=0,\ldots,L-1$), $\|\hat{\varphi}_{\theta,N,\xi+\half}^{(L^+)}\| \leq c$ uniformly in $\theta\in(0,\theta_0]$ for some $\theta_0 = \theta_0(c)>0$.

Thus, we get the expansion of $\hat{\varphi}_{N,\xi-\half}$ by multiplying \eqref{eq:varphi_tilde_expansion_1}--\eqref{eq:varphi_tilde_expansion_3} and \eqref{eq:varphi_denominator_scaling}.

It only remains to show validity of the representations \eqref{eq:varphi_hat_expansion_1}--\eqref{eq:varphi_hat_expansion_3}.
Indeed, \eqref{eq:varphi_hat_expansion_1} can be easily shown by tracking the leading order terms.
To show \eqref{eq:varphi_hat_expansion_3}, we use the representation \eqref{eq:inverse_expansion:repr} to express
\[
Z^{(L)}[\hat{\varphi}_{N,\xi+\half}^{(0)},\ldots,\hat{\varphi}_{\theta,N,\xi+\half}^{(L^+)}]
=
-\nu_{N,\xi-\half}^{-2} \big(P_{0,N,\xi}[\hat{\varphi}_{\theta,N,\xi+\half}^{(L^+)}], \tilde{\varphi}^{(0)}_{N,\xi-\half}\big)_{L^2}
+
\tilde{Z}^{(L)}[\hat{\varphi}_{N,\xi+\half}^{(0)},\ldots,\hat{\varphi}_{N,\xi+\half}^{(L-1)}]
\]
and note that the terms involving $\hat{\varphi}_{N,\xi+\half}^{(L)}$ (that cannot be adsorbed into $\hat{M}_{\theta,N,\xi}^{(L^{++})}$) sum up to
\begin{align*}
&
\nu_{N,\xi-\half}^{-1}P_{0,N,\xi}[\hat{\varphi}_{N,\xi+\half}^{(L)}] -\nu_{N,\xi-\half}^{-2} \big(P_{0,N,\xi}[\hat{\varphi}_{\theta,N,\xi+\half}^{(L^+)}], \tilde{\varphi}^{(0)}_{N,\xi-\half}\big)_{L^2}
\\=~&
{\Pr}_{\tilde{\varphi}^{(0)}_{N,\xi-\half}}  \hat{P}_{0,N,\xi}[\hat{\varphi}_{N,\xi+\half}^{(L)}]
=
\hat{Q}_{N,\xi}[\hat{\varphi}_{N,\xi+\half}^{(L)}]
.
\end{align*}
\eqref{eq:varphi_hat_expansion_2} is shown similarly.
\end{proof}

We now can finalize the proof for the leading order term:
\begin{proof}[Proof of \eqref{eq:phi_n_expansion_term_start}]
Easily follows from Lemma \ref{lem:interation_boundness:induction_basis}, \eqref{eq:varphi_hat_expansion_2} and \eqref{eq:P0_explicit} by induction.
\end{proof}

\begin{lemma}\label{lem:Qnorm}
There exist $0<\mu<1$ and $\xi_0>0$ such that $\|\hat{Q}_{N,\xi}\| \leq \mu$ for all $N$ and $\xi_0\leq\xi\leq N-\xi_0$.
\end{lemma}
\begin{proof}
Let
$p_{\infty}(x,y)
:=
e^{-x^2/2-\alpha x y-y^2/2}$,
$P_{\infty}[\varphi](y) := \int_{-\infty}^{\infty} p_\infty(x,y) \varphi(y) \dy$, $\nu_\infty := \big(\frac{2\pi}{1+\kappa}\big)^{1/2}$, and $\hat{P}_{\infty} := \frac{1}{\nu_{\infty}} P_{\infty}$.
Then, note that an argument similar to Proposition \ref{prop:P0_explicit} applied to $\hat{P}_{\infty}$ yields that
$
\hat{P}_{\infty}[\psi_\kappa](x) =
\psi_{\sigma'}
$
with $\sigma'=1-\frac{\alpha^2}{(1+\kappa)}=\kappa$.
That is, $\psi_\kappa$ is the eigenfunction of $\hat{P}_{\infty}$ with eigenvalue $1$.
Next, note that $\hat{P}_{\infty}$ is a compact self-adjoint operator with a positive kernel.
A standard result in the spectral theory of integral operators yields that its maximal eigenvalue has multiplicity one and it is the largest by modulus eigenvalue, and moreover the associated eigenfunction is everywhere positive. (This result follows, e.g., from \cite[Theorem 7.1.3]{Helffer2010spectral}.)

Hence $\psi_\kappa$ is such an eigenfunction (due to orthogonality of different eigenfunctions, there is a unique eigenfunction of norm one that is everywhere positive) and hence the space orthogonal to $\psi_\kappa$ is an invariant subspace of the operator $\hat{P}_{\infty}$ and its restriction on this subspace has norm less than $1$.
The latter statement implies that
$\|{\Pr}_{\psi_{\kappa}} P_{\infty}\| < 1$.

It remains to note that since we have just proved \eqref{eq:phi_n_expansion_term_start}, $\hat{Q}_{N,\xi} = {\Pr}_{\psi_{\sigma_{N,\xi-\half}}} \hat{P}_{0,N,\xi}$ converges to ${\Pr}_{\psi_{\kappa}} P_{\infty}$ as $a_{N,\xi}\to 1$, $b_{N,\xi}\to \alpha$, $c_{N,\xi}\to 1$, $\sigma_{N,\xi-\half}\to \kappa$ (thanks to Lemma \ref{lem:operator_norm_bound}), and the four latter statements hold as $K_0\to\infty$ thanks to \eqref{eq:ass:V_asymptote} and \eqref{eq:sigma_minus_kappa}.
\end{proof}

\begin{proof}[Proof of Lemma \ref{lem:interation_boundness}]
It remains to show \eqref{eq:phi_n_expansion_term_mid}.

We prove the first bound in \eqref{eq:phi_n_expansion_term_mid} by induction over $\ell \geq 1$.
That is, we assume that it holds for the smaller values of $\ell$.
Hence we can estimate, using Lemma \ref{lem:Qnorm},
\begin{align*}
\big\|\hat{\varphi}_{N,\xi-\half}^{(\ell)}\big\|
\leq~& \|\hat{Q}_{N,\xi} \big[\hat{\varphi}_{N,\xi+\half}^{(\ell)}\big]\big\|
+
\big\| \hat{M}_{N,\xi}^{(\ell)}\big[\hat{\varphi}_{N,\xi+\half}^{(0)},\ldots,\hat{\varphi}_{N,\xi+\half}^{(\ell-1)}\big] \big\|
\leq
\mu \big\|\hat{\varphi}_{N,\xi+\half}^{(\ell)}\big\|
+ C
\end{align*}
for $\xi_0\leq \xi\leq N-\xi_0$, where, since $\hat{M}_{N,\xi}^{(\ell)}$ is uniformly locally Lipschitz, $C$ does not depend on $N$, $\theta$, or $\xi$ by our induction assumption.

Hence Lemma \ref{lem:recursive_bound} gives us the bound
\[
\max_{K_0\leq n\leq N-K_0} \|\hat{\varphi}_{N,n}^{(\ell)}\|
\leq \max\big\{ \|\hat{\varphi}_{N,N-K_0}^{(\ell)}\|, \smfrac{C}{1-\mu}\big\}
\qquad\text{for $K_0 = \xi_0+\smfrac12$}
.
\]
Finally, notice that since $\|\hat{Q}_{N,\xi}\|$ are bounded even for $\xi>N-K_0$ (but not necessarily by $\mu<1$), and $\hat{\varphi}_{N,N}^{(\ell)}$ is bounded independently of $N$ (cf.\ Lemma \ref{lem:interation_boundness:induction_basis}), $\max_{K_0\leq n<N} \|\hat{\varphi}_{N,n}^{(\ell)}\|$ are also bounded independently of $N$.

Finally, to prove the second bound in \eqref{eq:phi_n_expansion_term_mid}, we will use the uniform local Lipschitz continuity of $\hat{M}_{\theta,N,\xi}^{(L^{++})}$ around $0$, i.e.,
\[
\hat{M}_{\theta,N,\xi}^{(L^{++})}[\hat{\varphi}_{N,\xi+\half}^{(0)},\ldots,\hat{\varphi}_{N,\xi+\half}^{(L-1)},\hat{\varphi}_{\theta,N,\xi+\half}^{(L^+)}]
\leq C_r \big\|\hat{\varphi}_{\theta,N,\xi+\half}^{(L^+)}\big\|
\qquad
(\text{for }\big\|\hat{\varphi}_{\theta,N,\xi+\half}^{(L^+)}\big\| \leq r,
~
0<\theta\leq \theta_0)
,
\]
where $\theta_0=\theta_0(r)$ and $C_r$ depend on $r$ which will be chosen later.
We estimate the first two terms in \eqref{eq:varphi_hat_expansion_3} similarly to the above and we get
\begin{align*}
\big\|\hat{\varphi}_{\theta,N,\xi-\half}^{(L^+)}\big\|
\leq ~&
\mu \, \big\|\hat{\varphi}_{\theta,N,\xi+\half}^{(L^+)}\big\|
+ C
+ \theta^{1/2} C_r \big\|\hat{\varphi}_{\theta,N,\xi+\half}^{(L^+)}\big\|
\\ = ~&
\big(\mu + \theta^{1/2} C_r\big) \, \big\|\hat{\varphi}_{\theta,N,\xi+\half}^{(L^+)}\big\|
+ C
\qquad\forall \xi\leq K_0+\smfrac12
.
\end{align*}
Next, pick some $\mu'\in(\mu,1)$ independently of $r$.
Then for $\theta^{1/2} \leq (\mu'-\mu)/C_r$, Lemma \ref{lem:recursive_bound} again applies and yields the bound $C_{L^+}$ in \eqref{eq:phi_n_expansion_term_mid} similarly as above.

It remains to notice that $C_{L^+}$ is independent of the choice of $r$, hence we can now choose $r=C_{L^+}$ and possibly decrease $\theta_0=\theta_0(r)$ to satisfy $\theta_0^{1/2} \leq (\mu'-\mu)/C_r$.
\end{proof}

\section{Proof for the Infinite Lattice}\label{sec:infinite}

In principle, the above results should be enough to establish convergence of a subsequence of $\mu_{\theta,N}$ and their expansion terms as $N\to\infty$.
However, our goal, set fourth by Theorem \ref{th:uniform_expansion}, is more ambitious: we want to prove that any possible sequence converges to a unique $\mu_{\theta,\infty}$ together with the respective expansion terms.

The structure of this section is as follows.
In Section \ref{sec:infinite:1} we establish the convergence of $\mu_{\theta,N}$ as $N\to\infty$, in Section \ref{sec:infinite:2}---the convergence of $\mu^{(0)}_{N}$, and in Section \ref{sec:infinite:3}---the convergence of other expansion terms.
Next, in Section \ref{sec:infinite:4} we show convergence of the inverted Hessians of the energy.
Finally, in Section \ref{sec:infinite:5} we combine all these results into a proof of the statements of Theorem \ref{th:uniform_expansion} for the infinite lattice.

We retain the following conventions from Section \ref{sec:finite_lattice}: $L=2M$ and $C$ denotes a generic constant independent of $N$ or $\theta$.
However, we will no longer omit tildes in the notations introduced in Section \ref{sec:finite_lattice:change}.

\subsection{Convergence of $\mu_{\theta,N}$}\label{sec:infinite:1}

\begin{lemma}\label{lem:mu_orig_conv}
For any $\theta>0$ there exists $\mu_{\theta,\infty} = \lim_{N\to\infty}\mu_{\theta,N}$.
\end{lemma}
\begin{proof}
Similarly to Proposition \ref{prop:expansion_reduction} (note that we now use tildes for the notations introduced in Section \ref{sec:finite_lattice:change}), we can express
\begin{displaymath}
\<\mu_{\theta,N}, A\>
=
\frac{%
	\int_{\calU_{K}}
	A(u)
	e^{-\theta^{-1} \bar{E}_K(u)}
	P^{\circ N-K}[p_\theta^{\bc}] (u_{K})\,
	P^{\circ N-K}[p_\theta^{\bc}] (u_{-K})
	\du
}{%
	\hphantom{A(u)}
	\int_{\calU_{K}}
	e^{-\theta^{-1} \bar{E}_K(u)}
	P^{\circ N-K}[p_\theta^{\bc}] (u_{K})\,
	P^{\circ N-K}[p_\theta^{\bc}] (u_{-K})
	\du
}
\end{displaymath}
for $A\in \C^{2M}(\calU_K)$, where
\begin{align}\notag
	\bar{E}_{K}(v)
	:=~&
	\calP(u) +
	\sum_{\xi=-K+\half}^{K-\half} V(u_{\xi-\half}, u_{\xi+\half})
,
\\ \notag
	P[\varphi](x)
	:=~&
	\int_{-\infty}^{\infty}
		e^{-\theta^{-1} V(x,y)} \varphi(y)
	\dy
	\quad \forall\theta>0
,
\\ \notag
	p_\theta^\bc(x)
	:=~&
	e^{-\theta^{-1} V_\theta^\bc(x,y)}
	\quad \forall\theta>0
.
\end{align}

We will use the earlier argument to conclude that the leading eigenfunction of $P_\theta$ is unique and positive.
This, however, requires some care.
Recall the definition of $\calA$ from \eqref{eq:orig_ass_X}.
The operator $\bar{P}_\theta := P_\theta|_{L^2(\calA \to \calA)}$ has a strictly positive kernel in the sense of \cite[Definition 7.1.2]{Helffer2010spectral} and hence we can argue as in \cite[Theorem 7.1.3]{Helffer2010spectral} to obtain a unique eigenvalue $\lambda_1$ with the largest absolute value corresponding to an almost everywhere positive eigenfunction $\psi:\calA\to\bbR$.
We extend it by zero outside $\calA$ and choose its scaling so that $\|\psi\|_{L^2}=1$.

Next, one can show that $\<P[p_\theta^\bc],\psi\>_{L^2} \ne 0$ (note that in general $p_\theta^\bc\notin L^2$)---a converse of this statement quickly leads to concluding that $V_\theta^\bc=+\infty$ on $\calA$, which contradicts \eqref{eq:orig_ass:Vbc_bound}.
Hence we have that
\[
\frac{\lambda^{-N+K+1} P^{\circ N-K}[p_\theta^\bc]}{\<P[p_\theta^\bc],\psi\>_{L^2}} \to \psi
\qquad\text{in $L^2$}
,
\]
as $N\to\infty$, and therefore
\begin{displaymath}
\<\mu_{\theta,N}, A\>
\to
\frac{%
	\int_{\calU_{K}}
	A(u)
	e^{-\theta^{-1} \bar{E}_{K}(u)}
	\psi(u_{K})\,
	\psi(u_{-K})
	\du
}{%
	\hphantom{A(u)}
	\int_{\calU_{K}}
	e^{-\theta^{-1} \bar{E}_{K}(u)}
	\psi(u_{K})\,
	\psi(u_{-K})
	\du
}
=
\<\mu_{\theta,\infty}, A\>,
\end{displaymath}
which is a well-defined limit upon noting the positivity of the denominator and the quadratic growth of $\bar{E}_{K}(u)$ at infinity (the latter follows, e.g., from Lemma \ref{lem:Enk_convex}).
The uniqueness of $\mu_{\theta,\infty}$ and its independence from $V_\theta^\bc$ is obvious.
\end{proof}

\subsection{Convergence of $\mu_{N}^{(0)}$}\label{sec:infinite:2}

We need to prove one more property of $\tilde{V}_{N,\xi}$ in addition to those in Lemma \ref{lem:ass}.
\begin{lemma}\label{lem:ass2}
\begin{align}
\label{eq:ass2:der_conv}
\nabla^{2+\ell} V(u^*_{N,\xi-\half},u^*_{N,\xi+\half}) \to
\nabla^{2+\ell} V(u^*_{\infty,\xi-\half},u^*_{\infty,\xi+\half}),
\qquad \forall\ell=0,\ldots,L+1
\end{align}
as $N\to\infty$ uniformly in $\xi$.
\end{lemma}
\begin{proof}
This follows immediately from the assumed regularity on $V$ and the convergence of $u^*_N$ to $u^*_\infty$.
\end{proof}

To establish convergence of $\mu_{N}^{(0)}$, we must establish convergence of
$
\hat{\varphi}_{N,K}^{(0)}
=
\psi_{\sigma_{N,K}}
$
(cf.\ \eqref{eq:phi_n_expansion_term_start}) for a fixed $K$ as $N\to\infty$.
It will be easy to reduce proving the convergence of $\hat{\varphi}_{N,K}^{(0)}$ to proving the convergence of $\sigma_{N,K}$ which we state in the following lemma
\begin{lemma}\label{lem:conv_sigma}
There exists a sequence $\sigma_{\infty,n} \geq \kappa^2/2$ such that
\begin{equation}\label{eq:sigma_limit}
\sup_{\substack{n\in\bbN \\ n \leq N/2}}|\sigma_{N,n}-\sigma_{\infty,n}| \to 0
\qquad\text{as }N\to\infty.
\end{equation}
\end{lemma}
\begin{proof}
The proof will be based on estimating $|\sigma_{N,n}-\sigma_{N',n}|$ as $N,N'\to\infty$.
We will work under a convention that the pair ($N$, $N'$) is ordered such that $N \leq N'$.

Denote $S_{N,\xi}(\sigma) := a_{N,\xi}-\frac{b_{N,\xi}^2}{c_{N,\xi}+\sigma}$, so that $\sigma_{N,\xi-\half} = S_{N,\xi}(\sigma_{N,\xi+\half})$, and note that
\[
\frac{\dd}{\dd \sigma} S_{N,\xi}(\sigma)
=
\frac{b_{N,\xi}^2}{(c_{N,\xi}+\sigma)^2}
\leq
\frac{b_{N,\xi}^2}{c_{N,\xi}^2}
\to \alpha^2
\quad\text{as }\xi,N\to\infty
.
\]
Hence there is a constant $\xi_0\in\bbN+\frac12$ such that $0<\frac{\dd}{\dd \sigma} S_{N,\xi}(\sigma)<\mu:=\frac{1+\alpha^2}{2}$ (note that $\mu<1$) for all $\sigma>0$ as long as $N \geq \xi \geq \xi_0$.

{\it Step 1 ($n\geq \xi_0-\frac12$).}

Note that for $N' \geq N \geq \xi \geq \xi_0$,
\begin{align*}
|\sigma_{N,\xi-\half} - \sigma_{N',\xi-\half}|
=~&
|S_{N,\xi}(\sigma_{N,\xi+\half}) - S_{N',\xi}(\sigma_{N',\xi+\half})|
\\=~&
\big|
	\big(S_{N,\xi}(\sigma_{N,\xi+\half})
	- S_{N,\xi}(\sigma_{N',\xi+\half})\big)
	+ \big(S_{N,\xi}(\sigma_{N',\xi+\half})
	- S_{N',\xi}(\sigma_{N',\xi+\half})\big)
\big|
\\ \leq~&
\mu
|\sigma_{N,\xi+\half}-\sigma_{N',\xi+\half}| +
\underbrace{\big|S_{N,\xi}(\sigma_{N',\xi+\half})
	- S_{N',\xi}(\sigma_{N',\xi+\half})\big|}_{=:
	r_{N,N',\xi}}.
\end{align*}
The uniform in $\xi$ convergence of $a_{N,\xi}$, $b_{N,\xi}$ and $c_{N,\xi}$ which follows from \eqref{eq:ass2:der_conv} applied with $\ell=0$, together with boundedness of $\sigma_{N',\xi+\half}$ established in Lemma \ref{lem:P0_recurrent}, can be used to easily prove
\[
\sup_{\substack{\xi\in\bbN+\half \\ \xi_0 \leq \xi \leq N }}|r_{N,N',\xi}| =: r_{N,N'} \to 0
\quad\text{as } N,N'\to\infty.
\]

Then, applying Lemma \ref{lem:recursive_bound}, we obtain
\[
|\sigma_{N,n} - \sigma_{N',n}|
\leq \mu^{N-n} |\sigma_{N,N} - \sigma_{N',N}|
+ \frac{r_{N,N'}}{1-\mu},
\]
which yields, upon recalling that $|\sigma_{N,N} - \sigma_{N',N}|$ is uniformly bounded thanks to \eqref{eq:sigma_est},
\[
\sup_{\xi_0-\half \leq n\leq N/2}
|\sigma_{N,n} - \sigma_{N',n}|
\to 0
\quad\text{as }N,N'\to\infty.
\]
Hence the pointwise in $n$ limit of $\sigma_{N,n}$ exists for $n \geq \xi_0-\half$ and hence we get
\[
\sup_{\substack{n\in\bbN \\ \xi_0-\half \leq n \leq N/2}}|\sigma_{N,n}-\sigma_{\infty,n}| \to 0
\qquad\text{as }N\to\infty.
\]

{\it Step 2 ($n<\xi_0-\frac12$).}

We only need to note that $S_{N,\xi}$ is locally Lipschitz for all $N$ and $\xi$ (although the Lipschitz constant may not be less than $1$ for $\xi<\xi_0$), therefore
$
\sup_{n<\xi_0-\half}
|\sigma_{N,n} - \sigma_{N',n}|
$
is still bounded by $C |\sigma_{N,\xi_0-\half} - \sigma_{N',\xi_0-\half}|\to 0$ for some constant $C=C(\xi_0)$.
This completes the proof of the existence of $\sigma_{\infty,n}$ and \eqref{eq:sigma_limit}.

The bound $\sigma_{\infty, n} \geq \kappa^2/2$ follows from the lower bound in \eqref{eq:sigma_est}.
\end{proof}

\begin{corollary}\label{corr:mu0_conv}
$\sup_{K \leq N/2} \|\hat{\varphi}_{N,K}^{(0)}-\psi_{\sigma_{\infty,K}}\| \to 0$
as $N\to\infty$.
\end{corollary}
\begin{proof}
Follows upon a standard application of the dominated convergence theorem.
\end{proof}

\subsection{Convergence of other $\mu_{N}^{(\ell)}$}\label{sec:infinite:3}

In this subsection we prove the induction step, formulated in the following lemma.
\begin{lemma}\label{lem:mu_ell_conv}
For any $\ell\in\{1,\ldots,L-1\}$, there exists a sequence $\hat{\varphi}_{\infty,K}^{(\ell)} \in L^2(\bbR)$ such that
\begin{equation}\label{eq:mu_ell_conv}
\sup_{K \leq 2^{-\ell-1} N} \|\hat{\varphi}_{N,K}^{(\ell)} - \hat{\varphi}_{\infty,K}^{(\ell)} \|_{L^2} \to 0
\quad\text{as } N\to\infty
,
\end{equation}
provided that the same holds for all smaller values of $\ell$.
\end{lemma}

To that end, using \eqref{eq:varphi_hat_expansion_2} and again adopting the convention that $N \leq N'$, estimate
\begin{align*}
\big\|
	\hat{\varphi}_{N,\xi-\half}^{(\ell)}
	-
	\hat{\varphi}_{N',\xi-\half}^{(\ell)}
\big\|
=~&
\Big\|
	\hat{Q}_{N,\xi} [\hat{\varphi}_{N,\xi+\half}^{(\ell)}]
	-
	\hat{Q}_{N',\xi} [\hat{\varphi}_{N',\xi+\half}^{(\ell)}]
	\\~&
	+
	\hat{M}_{N,\xi}^{(\ell)}[\hat{\varphi}_{N,\xi+\half}^{(0)},\ldots,\hat{\varphi}_{N,\xi+\half}^{(\ell-1)}]
	-
	\hat{M}_{N',\xi}^{(\ell)}[\hat{\varphi}_{N',\xi+\half}^{(0)},\ldots,\hat{\varphi}_{N',\xi+\half}^{(\ell-1)}]
\Big\|
\\ \leq ~&
	\big\|\hat{Q}_{N,\xi}
	\big(\hat{\varphi}_{N,\xi+\half}^{(\ell)} - \hat{\varphi}_{N',\xi+\half}^{(\ell)}\big)\big\|
	\\~&
	+
	\big\|
		\hat{Q}_{N,\xi}
		-
		\hat{Q}_{N',\xi}
	\big\|\,
	\big\|\hat{\varphi}_{N',\xi+\half}^{(\ell)}\big\|
	\\~&
	+
\big\|
	\hat{M}_{N,\xi}^{(\ell)}[\hat{\varphi}_{N,\xi+\half}^{(0)},\ldots,\hat{\varphi}_{N,\xi+\half}^{(\ell-1)}]
	-
	\hat{M}_{N,\xi}^{(\ell)}[\hat{\varphi}_{N',\xi+\half}^{(0)},\ldots,\hat{\varphi}_{N',\xi+\half}^{(\ell-1)}]
\big\|
	\\~&
	+
\big\|
	\hat{M}_{N,\xi}^{(\ell)}[\hat{\varphi}_{N',\xi+\half}^{(0)},\ldots,\hat{\varphi}_{N',\xi+\half}^{(\ell-1)}]
	-
	\hat{M}_{N',\xi}^{(\ell)}[\hat{\varphi}_{N',\xi+\half}^{(0)},\ldots,\hat{\varphi}_{N',\xi+\half}^{(\ell-1)}]
\big\|
\\ =: ~&
	{\rm T}_1+{\rm T}_2+{\rm T}_3+{\rm T}_4.
\end{align*}

\medskip \noindent {\it Bound on ${\rm T}_1$.}
$
{\rm T}_1
=
	\big\|\hat{Q}_{N,\xi}
	\big[\hat{\varphi}_{N,\xi+\half}^{(\ell)} - \hat{\varphi}_{N',\xi+\half}^{(\ell)}\big]\big\|
\leq
	\mu \big\|
	\hat{\varphi}_{N,\xi+\half}^{(\ell)} - \hat{\varphi}_{N',\xi+\half}^{(\ell)}
	\big\|.
$

\medskip \noindent {\it Bound on ${\rm T}_2$.}
By Lemma \ref{lem:Qnorm}, the operators $\hat{Q}_{N,\xi}$ and $\hat{Q}_{N',\xi}$ both approach ${\Pr}_{\psi_{\kappa}} \hat{P}_{\infty}$ as $N,N'\to\infty$ uniformly in $\xi<N/2$ ($\hat{P}_{\infty}$ is introduced in the proof of Lemma \ref{lem:Qnorm}).
This together with boundedness of $\big\|\hat{\varphi}_{N',\xi+\half}^{(\ell)}\big\|$ (which we proved in Section \ref{sec:finite_lattice}) proves that
\[
\sup_{\xi \leq N/2} {\rm T}_2 = o(1) + {\Pr}_{\psi_{\kappa}} \hat{P}_{\infty}\big[\hat{\varphi}_{N',\xi+\half}^{(\ell)}\big] - {\Pr}_{\psi_{\kappa}} \hat{P}_{\infty}\big[\hat{\varphi}_{N',\xi+\half}^{(\ell)}\big] \to 0
\quad \text{as } N,N'\to\infty.
\]

\medskip \noindent {\it Bound on ${\rm T}_3$.}
The fact that $\sup_{\xi \leq 2^{-\ell} N} {\rm T}_3 \to 0$, as $N,N'\to\infty$ follows from the uniform local Lipschitz regularity of $\hat{M}_{N,\xi}^{(\ell)}$ and boundedness of differences of its arguments for $\xi \leq 2^{-\ell} N$ (as provided by the induction assumption).

\medskip \noindent {\it Bound on ${\rm T}_4$.}
The bound on ${\rm T}_4$ is equally straightforward, but is more lengthly to derive.
It is derived in the following five steps whose details we omit:
\begin{itemize}
\item[(i)] Note that $\sigma_{N,\xi}$, $\psi_{\sigma_{N,\xi}}$, and the derivatives of $V$ evaluated at $(u^*_{N,\xi-\half},u^*_{N,\xi+\half})$ converge as $N\to\infty$ uniformly in $\xi<N/2$. This implies, in particular, that $p_{0,N,\xi} \to p_{\infty}$ uniformly in $\xi<N/2$.

\item[(ii)] The coefficients of polynomials $q_{N,\xi}^{(m)}(x,y)$ defined in the proof of Lemma \ref{lem:varphi_tilde_expansion} are linear combinations of the derivatives of $V$ evaluated at $(u^*_{N,\xi-\half},u^*_{N,\xi+\half})$ (cf.\ Lemma \ref{lem:Gaussian_expansion}(a)) and therefore $q_{N,\xi}^{(m)} p_{0,N,\xi}$ converge in $\C_{\tilde{\gamma}_V - \eps}(\bbR^2)$.
Hence $P_{0,N,\xi}^{(m)}$ converge uniformly in $\xi$ as provided by Lemma \ref{lem:operator_norm_bound}.

\item[(iii)] The linear operators $M_{N,\xi}^{(m)}$, defined in Lemma \ref{lem:varphi_tilde_expansion}, likewise converge uniformly in $\xi$ as $N\to\infty$.

\item[(iv)] The (nonlinear) operators $Z^{(\ell)}$ (see the proof of Lemma \ref{lem:elimination_operators}) were constructed from continuous operators and therefore are themselves continuous, once their arguments are restricted to a bounded region in $L^2$.

\item[(v)] Finally this implies that $\hat{M}_{N,\xi}^{(m)}$ converge as $N\to\infty$ uniformly in $\xi$ and their arguments, once they are constrained to a bounded region.
\end{itemize}

\medskip \noindent {\it Finalizing the estimate.}
Combining the bounds on all ${\rm T}_i$, we obtain
\begin{align*}
\big\|
	\hat{\varphi}_{N,\xi-\half}^{(\ell)}
	-
	\hat{\varphi}_{N',\xi-\half}^{(\ell)}
\big\|
\leq
\mu
\big\|
	\hat{\varphi}_{N,\xi+\half}^{(\ell)}
	-
	\hat{\varphi}_{N',\xi+\half}^{(\ell)}
\big\|
+ R_{N,N',\xi},
\end{align*}
where $R_{N,N',\xi}$ converges to $0$ uniformly in $\xi \leq 2^{-\ell} N$.
Applying Lemma \ref{lem:recursive_bound} to the last inequality, similarly as in the proof of Lemma \ref{lem:conv_sigma}, yields
$
\big\|
	\hat{\varphi}_{N,n}^{(\ell)}
	-
	\hat{\varphi}_{N',n}^{(\ell)}
\big\|
\to 0
$
uniformly in $n \leq 2^{-\ell-1} N$.

\subsection{Convergence of $H_N^{-1}$}\label{sec:infinite:4}

The final ingredient of the proof for the infinite lattice is the following lemma.
\begin{lemma}\label{lem:HNinv_conv}
For $K\in\bbN$,
$H_N^{-1}|_{\calU_K\to\calU_K} \to H_\infty^{-1}|_{\calU_K\to\calU_K}$ as $N\to\infty$.
\end{lemma}

\begin{proof}[Proof of Lemma \ref{lem:HNinv_conv}]
Due to equivalence of the weak and strong convergence in finite dimensions, it is sufficient to prove $H_N^{-1}|_{\calU_K\to\calU_K} \weakto H_\infty^{-1}|_{\calU_K\to\calU_K}$.
To that end, let $f\in\calU_K$ and consider the following two problems for $v_N$ and $v_\infty$,
\begin{align}
\<\ddel E_{0,N}(u^*_N) v_N, w\>
=~&
(f, w)_{\calU_N}
\quad \forall w\in\calU_N
,
\\
\<\ddel E_\infty(u^*_\infty) v_\infty, w\>
=~&
(f, w)_{\calU_\infty}
\quad \forall w\in\calU_\infty
.
\end{align}
(We note that $\ddel E_\infty(u^*_\infty)$ is coercive, thanks to \eqref{eq:defect-global}.)
The assertion of the lemma then follows from the estimate
\begin{equation}\label{eq:HNinv_conv_to_prove}
\|v_N - (v_\infty|_{\calU_N})\| \leq o(1).
\end{equation}

We use the standard consistency-stability argument:
\begin{align*}
\<\ddel E_{0,N}(u^*_N) (v_N - v_\infty|_{\calU_N}), w_N\>
=~&
(f, w_N) -
\<\ddel E_{0,N}(u^*_N) (v_\infty|_{\calU_N}), w_N\>
\\=~&
\<\ddel E_\infty(u^*_\infty) v_\infty, w_N\>
-
\<\ddel E_{0,N}(u^*_N) (v_\infty|_{\calU_N}), w_N\>
\\=~&
	\<\ddel E_\infty(u^*_\infty) v_\infty, w_N\>
	-
	\<\ddel E_{0,N}(u^*_\infty|_{\calU_N}) (v_\infty|_{\calU_N}), w_N\>
	\\~&+
	\<(\ddel E_{0,N}(u^*_\infty|_{\calU_N}) - \ddel E_{0,N}(u^*_N)) (v_\infty|_{\calU_N}), w_N\>
\\=:~& {\rm T}_1 + {\rm T}_2.
\end{align*}
Thanks to $\|u^*_N - (u^*_\infty|_{\calU_N})\| \to 0$ and regularity of $E_{0,N}$, we have that
\[
|{\rm T}_2| \lesssim \|(u^*_\infty|_{\calU_N}) - u^*_N\|_{\ell^\infty} \, \|v_\infty|_{\calU_N}\| \, \|w_N\|
\leq
o(1)
\|v_\infty\| \, \|w_N\|
.
\]

Next, we note that ${\rm T}_1$ contains only the boundary terms (note that $w_N\in\calU_N$) and we estimate them using the decay of $v_{\infty,n}$ as $n\to\pm\infty$:
\begin{align*}
|{\rm T}_1|
=~&
	\bigg|
	\sum_{\xi\in\{-N-\half,N+\half\}}
	\ddel V(u^*_{\infty,\xi-\half}, u^*_{\infty,\xi+\half})
			(v_{\infty,\xi-\half}, v_{\infty,\xi+\half}),
			(w_{N,\xi-\half}, w_{N,\xi+\half})
		\>
	\bigg|
\\ \leq ~&
	C \|w\|\,
	\sum_{n\in\{-N-1,-N,N,N+1\}}
	|v_{\infty,n}|
	= o(1) \|w_N\|
.
\end{align*}

Finally, employing the stability of $\ddel E_{0,N}(u^*_N)$ which follows from \eqref{eq:defect_stab} and $\|u^*_N-u^*_\infty|_{\calB_N}\| \to 0$, we get \eqref{eq:HNinv_conv_to_prove}.
\end{proof}

\subsection{Finalizing the Proof of Theorem \ref{th:uniform_expansion}}\label{sec:infinite:5}

\noindent {\it Proof of (b).}
Lemma \ref{lem:mu_orig_conv} guarantees the convergence of the left-hand side of \eqref{eq:expansion_main} for all $\theta>0$.
Corollary \ref{corr:mu0_conv} together with Lemma \ref{lem:conv_sigma} prove the convergence of $\mu_{N}^{(0)}$, and Lemma \ref{lem:mu_ell_conv} then yields the convergence of all $\mu_{N}^{(\ell)}$, $\ell<L$.
Therefore, the remainder in \eqref{eq:expansion_main}, $\theta^{L/2} \mu_{\theta,N}^{(L/2^+)}$, also convergences for any fixed $\theta>0$.

We remark that the above argument does not yield the uniform convergence of $\mu_{\theta,N}^{(L/2^+)}$ in $\theta$, which we expect to be true, possibly under additional regularity assumptions.
However, this will not be important for the applications we consider.

\medskip \noindent {\it Proof of (c, $N=\infty$).}
This follows directly from (c, $N<\infty$) proved in Section \ref{sec:finite_lattice}.

\medskip \noindent {\it Proof of (d, $N=\infty$).}
Thanks to the convergence of $u^*_N$ on $\calU_K$ for any fixed $K$, $A(u^*_N)$ and all terms in \eqref{eq:mu1_explicit} except for $H_N^{-1} = (\ddel E_{0,N}(u^*_N))^{-1}$ converge as $N\to\infty$, namely
\begin{align*}
\dddel E_{0,N}(u_N^*) \to \dddel E_\infty(u_\infty^*)
,\qquad
\del^k A(u_N^*) \to \del^k A(u_\infty^*),
\quad k=0,1,2
,
\end{align*}
and the convergence of $H_N^{-1}$ is established in Lemma \ref{lem:HNinv_conv}.
\qed

\section{Discussion and Conclusion}\label{sec:discussion-and-conclusion}

We presented a theory of calculation of a crystalline defect at finite temperature and illustrated how to use this theory to compare the performance of different computational methods.

The theory is rigorously justified in 1D for the free boundary conditions.
We claim that the techiques used in this work can be extended to wider interaction range or Dirichlet boundary conditions.
However, it is not clear to us how difficult it would be to extend this theory to two or three dimensions; nevertheless, we see no reasons why such a theory would fail in those cases.
Therefore, we propose to use this theory in the following way.
One could simply start by \emph{conjecturing} that a result similar to Theorem \ref{th:uniform_expansion} holds in a particular situation.
Such a conjecture would yield a tool to quantify the errors in approximating the respective Gibbs measures for different methods and propose better methods in a more systematic way.

In particular, in dimension two or three, the structure of infinite-dimensional Gibbs measures seems much more complicated than in one dimension, while the main difficulty in analyzing $\<\mu_{\theta,\infty}^{(1)}, A\>$ would be to deal with the inverse of Hessian operators that by now are quite well understood in the context of the zero-temperature theory.

Furthermore, in dimension two or three, we expect that the hot-QC method would yield a qualitative advantage over a simple atomistic calculation as opposed to the 1D case.
Indeed, in 1D a perturbation generated by a boundary condition decays exponentially with the same rate for different methods (namely, $\sim \lambda^n$).
However perturbations in higher dimensions decay algebraically and, at least in the zero-temperature case, different methods have errors that decay with different algebraic rates \cite{EhrlacherOrtnerShapeev2013preprint, LuskinOrtner2013acta}.

\appendix
\section{Expansions}\label{sec:expansion_proofs}

This section proves expansion of Laplace-like integrals and their bounds in $\bbR^n$.
Note that the constant in such bounds depend on $n$, but we will not always mention this dependence.

\subsection*{Proof of Lemma \ref{lem:Gaussian_expansion}}
In this proof $C$ denotes a generic constants that may depend only on $L$, $\gamma$, and $c$.

{\it Proof of (a).}
Note that \eqref{eq:Gaussian_expansion:Ezero} implies that $\del E^{(0)}(0)=0$.

The Taylor expansion of $E^{(m)}(v)$ yields
\[
\bigg|\theta^{m-1} E^{(m)}(v) -
\theta^{m-1} \sum_{k=0}^{M-1} P^{(m,k)}(v)
\bigg|
\leq
C \theta^{m-1} |v|^M
\qquad \forall v:|v|\leq r
\]
where we let $M = M_m = 3-2m+L$ and where $P^{(m,k)}(v)$ are the respective homogeneous polynomials of degree $k$, $P^{(0,2)}(v) = \frac12\<H v,v\>$ and $P^{(1,0)}(v) = E^{(1)}(0)$.

Next, we restrict $v$ to the ball $\|v\| \leq \theta^{-\alpha}$, where we choose $0<\alpha\leq 1/(6L+6)$.
Then a simple calculation shows that $\theta^{m-1} P^{(m,k)}(\sqrt{\theta} v)$, in the case $m-1+\frac{k}{2}>0$, can be bounded by $C \theta^{-\mu(m-1+k/2)}$ and $\theta^{m-1} |\sqrt{\theta} v|^M$ can be estimated by $C \theta^{-\mu (L/2+\frac12)}$ with $\mu=\frac{1}{L+1}$.

Hence we can combine the above expansions into
\begin{equation}\label{eq:expansion:preexp}
\theta^{-1} E_\theta(\sqrt{\theta} v) -
\smfrac12 \<H v, v\> + E^{(1)}(0) - E^{(1+L/2^+)}_\theta(\sqrt{\theta} v)
=
\sum_{\ell=1}^{L} \theta^{\ell/2} P^{(\ell)}(v) +
\theta^{L/2+\frac12} P_\theta^{(1+L^+)}(v),
\end{equation}
where $P^{(\ell)}$ are polynomials with coefficients being linear combinations of $\del^k E^{(\ell)}(0)$ and
\[
\sup_{|v| \leq \theta^{-\alpha}}
|P^{(\ell)}(v)\big| \leq C \theta^{-\mu\frac{\ell}{2}},
\qquad
\sup_{|v| \leq \theta^{-\alpha}}
\big|P^{(1+L^+)}_{\theta}(v)\big|
\leq C \theta^{-\mu({L/2+\frac12})}
.
\]
Hence one can apply Lemma \ref{lem:composition_expansion} with $z=\sqrt{\theta}$, $g(\varphi)=e^{-\varphi}$, and $f$ given by \eqref{eq:expansion:preexp} to obtain
\begin{align} \notag
\tilde{Q}_{\theta}(v)
:=~&
e^{-\theta^{-1} [E_\theta(\sqrt{\theta} v)
]
}
\,
e^{\frac12 \<H v, v\>}
\,
e^{E^{(1)}(0)}
\,
e^{\theta^{L/2} E^{(1+L/2^+)}(\sqrt{\theta}v)}
\\=~& \label{eq:Gaussian_expansion:exp}
\sum_{\ell=0}^{L} \theta^{\ell/2} \tilde{Q}_u^{(\ell)}(v)
+ \theta^{L/2+\frac12} \tilde{Q}_{\theta,u}^{(1+L^+)}(v)
,
\end{align}

Notice that $\tilde{Q}_u^{(\ell)}(-v) = (-1)^\ell \tilde{Q}_u^{(\ell)}(v)$ are some polynomials (follows from the fact that $P^{(\ell,k)}(v)$ are polynomials with a similar property), and $\tilde{Q}_u^{(0)} \equiv 1$ follows from tracking the leading-order term in the respective Taylor expansions.
Finally, \eqref{eq:composition_residual_est} gives the bound
\[
\big|\tilde{Q}^{(1+L^+)}_{\theta,u}(v)\big|
\leq C \theta^{\frac12} \theta^{-\mu(L/2+\frac12)} = C
\qquad \forall v:|v| \leq \theta^{-\alpha}
.
\]
We next expand the last term in the left-hand side of \eqref{eq:Gaussian_expansion:exp} and notice that $\theta^{L/2} E^{(1+L/2^+)}(\sqrt{\theta}v)$ is bounded by $\theta^{L/2}$, thanks to \eqref{eq:Gaussian_expansion:remainder_est}.
Dividing by $e^{E^{(1)}(0)}$ yields \eqref{eq:Gaussian_expansion_predef}, and the rest of the properties in part (a) are obvious.

{\it Proof of (b).}
Note that all $G^{(\ell)}$ are polynomials whose coefficients are bounded by $C$, hence \eqref{eq:Gaussian_expansion:terms} follows.

Next, note
\[
|G_\theta^{(L^+)}(v)|
\leq C
\qquad \forall v:|v| \leq \theta^{-\alpha}
.
\]
Hence one can show \eqref{eq:Gaussian_expansion:remainder} for $|v| \leq \theta^{-\alpha}$:
\[
\big|G_\theta^{(L^+)}(v) e^{-\frac12\<H v,v\>}\big|
\leq
C e^{-\frac{\gamma}{2} \|v\|^2}
\leq
C e^{-\frac{\gamma-\eps}{2} \|v\|^2}
\qquad \forall v:|v| \leq \theta^{-\alpha}
.
\]

Now consider $|v| > \theta^{-\alpha}$.
Then we can estimate the remainder directly from \eqref{eq:Gaussian_expansion_def}:
\begin{align*}
|G_\theta^{(L^+)}(v) e^{-\frac12\<H v,v\>}|
\leq~&
\theta^{-L/2} e^{-\theta^{-1} E_\theta(\sqrt{\theta} v)}
+
\sum_{\ell=0}^{L-1} \theta^{\ell/2-L/2} \big|G^{(\ell)}(v)\big| e^{-\frac12\<H v,v\>}
\end{align*}
by noting that the polynomials $G^{(\ell)}(v)$ can be estimated by $e^{\frac{\eps}{4} \|v\|^2}$, and hence
\begin{align*}
|G_\theta^{(L^+)}(v) e^{-\frac12\<H v,v\>} e^{\frac{\gamma-\eps}{2} \|v\|^2}|
\leq~&
C \theta^{-L/2} e^{-\frac{\eps}{2} \|v\|^2}
+
C \sum_{\ell=0}^{L-1} \theta^{\ell/2-L/2} e^{-(\frac{\eps}{2}-\frac{\eps}{4}) \|v\|^2}
\\ \leq~&
C \theta^{-L/2} e^{-\frac{\eps}{2} \theta^{-2\alpha}}
+
C \sum_{\ell=0}^{L-1} \theta^{\ell/2-L/2} e^{-\frac{\eps}{4} \theta^{-2\alpha}}
\qquad \forall v:|v| > \theta^{-\alpha}
\end{align*}
which approaches zero exponentially as $\theta\to 0$, and hence bounded uniformly in $\theta \in (0,\theta_0]$. \qed

\subsection*{Proof of Lemma \ref{lem:expansion_generic}}

{\it Proof of (a) and (b).}
Lemma \ref{lem:Gaussian_expansion} gives us the expansion \eqref{eq:Gaussian_expansion_def}.
We also use the expansion for $F_\theta(v)$ and
\[
A(\sqrt{\theta} v) =
	\sum_{\ell=0}^{L-1}
		\smfrac{1}{\ell !}
		\theta^{\ell/2}\,
		\del^\ell\! A(0)[v^{\otimes\ell}]
		+
		\theta^{L/2}\,A_\theta^{(L^+)}(v),
\]
where each term, $\del^\ell\! A(0)[v^{\otimes\ell}]$, and the remainder, $A_\theta^{(L^+)}(v)$, have at most polynomial growth at infinity at a rate bounded by $\|A\|_{\C^{L}}$.
Multiplying these expansions yields
\begin{align*}
\hat{F}_\theta(u)
=~& F_\theta(u) e^{-\theta^{-1} E_\theta(\sqrt{\theta}u)}
\\
=~&
\sum_{\ell=0}^{L-1} \theta^{\ell/2}
	\sum_{k=0}^{\ell} F^{(k)}(v) Q^{(\ell-k)}(v) e^{-\frac12 \<H v,v\>}
	+ \theta^{L/2} \hat{F}_\theta^{(L^+)}(v)
\\ =:~&
\sum_{\ell=0}^{L-1} \theta^{\ell/2} \hat{F}^{(\ell)}(v)	+ \theta^{L/2} \hat{F}_\theta^{(L^+)}(v)
,\qquad\qquad\text{and}
\\
\tilde{F}_\theta(A; u)
=~& A(\sqrt{\theta} v) \hat{F}_\theta(u)
\\
=~&
\sum_{\ell=0}^{L-1}
	\smfrac{1}{\ell !}
	\theta^{\ell/2}
	\sum_{k=0}^{\ell} \del^{\ell} A(0)[v^{\otimes\ell}] \hat{F}^{(\ell-k)}(v)
	+ \theta^{L/2} \tilde{F}_\theta^{(L^+)}(A; v)
\\ =:~&
\sum_{\ell=0}^{L-1} \theta^{\ell/2} \tilde{F}^{(\ell)}(A; v)	+ \theta^{L/2} \tilde{F}_\theta^{(L^+)}(A; v)
.
\end{align*}
It should also be noted that $\hat{F}^{(\ell)}(-v) = (-1)^{\ell} \hat{F}^{(\ell)}(v)$ and $\tilde{F}^{(\ell)}(A; -v) = (-1)^{\ell} \tilde{F}^{(\ell)}(A; v)$.

One can then apply Lemma \ref{lem:Gaussian_expansion}(b) with $\eps = \frac{\gamma-\gamma'}{4}$ to bound uniformly in $\theta$:
\begin{align*}
\|\hat{F}^{(\ell)}\|_{L^1}
\leq
C \|\hat{F}^{(\ell)}\|_{L^2_{-\eps}}
\leq
C \sum_{k=0}^{\ell}
	\|F^{(k)}\|_{L^2_{-\gamma'}}
	\|Q^{(\ell-k)} e^{-\frac12 \<H \bullet,\bullet\>}\|_{L^2_{\gamma-2\eps}}
\leq C,
\end{align*}
where we used \eqref{eq:Gaussian_expansion:terms} in the last step.
Similarly one can show $\|\hat{F}_\theta^{(L^+)}\|_{L^1} \leq C$.

We can likewise bound the terms in the numerator,
\[
\|\tilde{F}^{(\ell)}(A)\|_{L^1} \leq \|A\|_{\C^{L}} \, C,
\quad
\|\tilde{F}_\theta^{(L^+)}(A)\|_{L^1} \leq \|A\|_{\C^{L}} \, C,
\]
where $C$ depends on $\gamma$, $\gamma'$, $L$, and $c$.
Hence the expansion of both the denominator and numerator of \eqref{eq:expansion_generic:def} exists,
\begin{equation}\label{eq:muEF}
\<\mu_{\theta,E,F}, A \>
=
	\frac{%
		\tilde{f}_\theta(A)
	}{%
		\hat{f}_\theta
	}
=
	\frac{%
		\sum_{\ell=0}^{L-1} \theta^{\ell/2} \tilde{f}^{(\ell)}(A) + \theta^{L/2} \tilde{f}_\theta^{(L^+)}(A)
	}{%
		\sum_{\ell=0}^{L-1} \theta^{\ell/2} \hat{f}^{(\ell)} + \theta^{L/2} \hat{f}_\theta^{(L^+)}
	}
,
\end{equation}
with $\hat{f}^{(\ell)} = \int_{\bbR^n} \hat{F}^{(\ell)} \dv$, $\tilde{f}^{(\ell)} = \int_{\bbR^n} \tilde{F}^{(\ell)} \dv$ and $\hat{f}^{(L^+)}_\theta,\tilde{f}^{(L^+)}_\theta$ defined similarly, where all terms are uniformly bounded above by, respectively, $C$ and $\|A\|_{\C^{L}} \, C$.

To show uniform boundedness of $\hat{f}_\theta^{-1}$, we first note that $\|\ddel E^{(0)}\| \leq c$ (this is a consequence of \eqref{eq:Gaussian_expansion:Eell_norm}) yields a bound $\|H\| \leq c$.
This together with the second bound in \eqref{eq:expansion_generic:first_term_ass} yields a uniform bound on $\hat{f}^{(0)}$ above $0$:
\[
\hat{f}^{(0)} =
\int_{\bbR^n} F^{(0)}(u) e^{-E^{(1)}(0)} e^{-\frac12 \<H u, u\>} \du
\geq
\int_{B_{c^{-1}}} c^{-1} e^{-E^{(1)}(0)} e^{-\frac12 c \|u\|^2} \du
>0
.
\]
Hence Lemma \ref{lem:inverse_expansion} guarantees existence of the expansion of $\hat{f}_\theta^{-1}$, and uniform boundedness of the terms, for $\theta \leq \theta_0$ for some $\theta_0 = \theta_0(\gamma, \gamma', L, c)$, where the bound also depends only on $\gamma$, $\gamma'$, $L$, and $c$.
Hence the expansion terms and remainder in $\tilde{f}_\theta(A) \hat{f}_\theta^{-1}$ are likewise bounded.

Finally, it is obvious that each term of $\tilde{f}_\theta(A) \hat{f}_\theta^{-1}$ depends linearly on $A$, and noticing that the odd terms cancel, concludes the proof of parts (a) and (b).

{\it Proof of (c).}
The fact that $\mu_{E,F}^{(m)}$ is a linear combination of Dirac delta and its derivatives follows from the explicit expressions of $\tilde{F}^{(\ell)}$.
Next,
\[
\<\mu_{E,F}^{(0)}, A\>
= \frac{\tilde{f}^{(0)}}{\hat{f}^{(0)}}
= \frac{A(0) \hat{f}^{(0)}}{\hat{f}^{(0)}}
= A(0).
\]
Finally, the expression for $\<\mu_{E,F}^{(1)}, A\>$ follow from Lemma \ref{lem:expansion_aux} which we formulate and prove below.
Indeed, Lemma \ref{lem:expansion_aux} gives the expressions for $\tilde{f}^{(2)}$ and $\hat{f}^{(2)}$ (the latter can be retrieved from the former by setting $A\equiv 1$) and we get
\begin{align*}
\<\mu_{E,F}^{(1)}, A\>
=~&
\frac{\tilde{f}^{(2)}}{\hat{f}^{(0)}} - \frac{\tilde{f}^{(0)} \hat{f}^{(2)}}{(\hat{f}^{(0)})^2}
\\=~&
\Big[
G(E^{(0)},E^{(1)},E^{(2)}) A(0)
+
\smfrac12 \ddel A(0) \!:\! H^{-1}
\\&
~-
\del A(0) \cdot H^{-1} \big(
	\del E^{(1)}(0)+\smfrac12\dddel E^{(0)}(0) \!:\! H^{-1}
\big)
\Big]
\\&
-
\Big[A(0) G(E^{(0)},E^{(1)},E^{(2)})\Big],
\end{align*}
which upon canceling similar terms concludes the proof of (c).

{\it Proof of (d).}
We have that for $\ell=0, 1, \ldots, L-1$,
\[
\bar{f}^{(\ell)}
= \int \hat{F}^{(\ell)} (v) \dv
= \sum_{k=0}^{\ell} \int F^{(k)}(v) Q^{(\ell-k)}(v) e^{-\frac12 \<H v,v\>} \dv,
\]
which is clearly a continuous function of $F^{(k)} \in L^2_{-\gamma'}$ and $Q^{(k)}(v) e^{-\frac12 \<H v,v\>}$ ($k \leq \ell$), and the latter are continuous functions of $\del^k E^{(m)}(0)$ ($k+2 m\leq 2\ell+2$).
Likewise, for a fixed $A$, each $\tilde{f}^{(\ell)}$ is a continuous function of the same arguments.
The stated continuity of \eqref{eq:muEF} hence follows.
\qed

\begin{lemma}\label{lem:expansion_aux}
In the notations of Lemma \ref{lem:expansion_generic} let $F(\theta, u) \equiv 1$ and assume $L>2$.
Then
\begin{align} \notag
\frac{\tilde{f}^{(2)}}{\hat{f}^{(0)}}
=~&
G(E^{(0)},E^{(1)},E^{(2)}) A(0)
+
\smfrac12 \ddel A(0) \!:\! H^{-1}
\\&
-
\del A(0) \cdot H^{-1} \big(
	\del E^{(1)}(0)+\smfrac12\dddel E^{(0)}(0) \!:\! H^{-1}
\big)
.
\label{eq:lem:expansion_aux_g2}
\end{align}
where $G$ is some function of the respective tensors.
\end{lemma}
\begin{proof}

We have that
\begin{align*}
\tilde{f}^{(2)} =~&
\int_{\bbR^n} \tilde{F}^{(2)}(u) e^{-E^{(1)}(0)-\frac12 \<H u,u\>} \du.
\end{align*}
To calculate $\tilde{F}^{(2)}(u)$, we introduce tensors
\[
e_{i_1,\ldots,i_k}^{(\ell,k)} := \frac{\partial^k}{\partial u_{i_1}\ldots\partial u_{i_k}} E^{(\ell)}(0)
\qquad\text{and}\qquad
a_{i_1,\ldots,i_k}^{(k)} := \frac{\partial^k}{\partial u_{i_1}\ldots\partial u_{i_k}} A(0)
.
\]

A direct calculation (in which we adopt the repeated index summation convention) yields
\begin{align*}
\tilde{F}^{(2)}(u)
=~&
N(u) +
R(E^{(0)},E^{(1)},E^{(2)},u)
a^{(0)}
+
\smfrac12 a_{i,j}^{(2)} u_i u_j
\\&-
a_{i}^{(1)} e_{j}^{(1,1)} u_i u_j
-
\smfrac16
a_{i}^{(1)} e_{j,k,m}^{(0,3)} u_i u_j u_k u_m
\end{align*}
where $N(u)$ is some odd function of $u$, and $R(E^{0},E^{1},E^{2},u)$ is a function of $u$ that has a form of products of tensors $e_{i_1,\ldots,i_k}^{(\ell,k)}$ for $\ell=0,1,2$ and a vector $u$.

Another direct calculation consisting of (i) a linear change of variables diagonalizing $\frac12 \<H u,u\>$, (ii) taking integrals, and (iii) contracting the tensor products yields \eqref{eq:lem:expansion_aux_g2}.
\end{proof}

\section{Composition of series}\label{sec:composition}

\begin{lemma}\label{lem:composition_expansion} 
Let $f=f(z)$ be a polynomial
\begin{equation}\label{eq:composition_expansion_original}
f(z) = \sum_{\ell=0}^{L+1} z^{\ell} f^{(\ell)},
\end{equation}
for some $L\in\bbN$, let $g(\varphi)$ be an analytic function in $\{\varphi\in\bbC : |\varphi-f^{(0)}|<R\}$ and assume that $|f(z)-f^{(0)}| < R$ for $z\in[0,z_0]$.
\begin{itemize}
\item[(a)]
$G(z) := g(f(z))$ is an analytic function on $[0,z_0]$ with the following finite Taylor expansion,
\begin{equation}\label{eq:composition_expansion}
G(z) = \sum_{\ell=0}^{L-1} z^{\ell} G^{(\ell)}[f^{(1)},\ldots,f^{(\ell)}] + z^{L} G^{(L^+)}[f^{(1)},\ldots,f^{(L)};z]
,\quad \forall z<z_0,
\end{equation}
where $G^{(0)} = g(f^{(0)})$, $F^{(\ell)}$ are some polynomials and $F^{(L^+)}$ is some analytic function of its $L+1$ arguments.

\item[(b)] If for some $c>0$, $\alpha\geq 0$, $0<\mu\leq \mu_0$, there holds $|f^{(\ell)}| \leq c \mu^{-\ell \alpha}$ for all $\ell\geq 1$ then
\begin{align} \label{eq:composition_term_est}
\big|G^{(\ell)}[f^{(1)},\ldots,f^{(\ell)}]\big| \leq~& C \mu^{-\ell \alpha}
\qquad\ell=0,\ldots,L-1
\\ \label{eq:composition_residual_est}
\sup_{0<z\leq z_1}\big|G^{(L^+)}[f^{(1)},\ldots,f^{(L)};z]\big| \leq~& C \mu^{-L \alpha}
\end{align}
hold for some $C>0$ and $z_1>0$ that depend on $g$, $L$, $c$, $R$, and $\mu_0^\alpha$.
In particular, $F^{(L^+)}$ is of the class $\C^m$ for any $m\in\bbZ_+$ on the set $\max_{1\leq \ell\leq L} |f^{(\ell)}| \leq c \mu^{-\ell \alpha}/2$ and $|z|\leq z_1/2$.

\end{itemize}
\end{lemma}

\begin{proof} 

Without loss of generality, assume $f^{(0)}=0$.

{\it Proof of (a).}
Analyticity of $g(f(z))$ is obvious.
Hence, using the expansion
\[
g(\varphi) = \sum_{j=0}^{\infty} g^{(j)} \varphi^j
,\quad \forall \varphi:|\varphi|<R
\]
where $g^{(j)} = \frac{1}{j!} \del^j g(0)$
we can expand
\begin{align*}
G(z) =~&
g\bigg(
	\sum_{i=1}^{L} z^{i} f^{(i)}
\bigg)
~=~
\sum_{j=0}^{\infty} g^{(j)} \bigg(
	\sum_{i=1}^{L} z^{i} f^{(i)}
\bigg)^j
~=~
\sum_{j=0}^{\infty} g^{(j)}
\prod_{k=1}^j
\bigg(
	\sum_{i=1}^{L} z^{i} f^{(i)}
\bigg)
\\=~&
\sum_{j=0}^{\infty} g^{(j)}
\sum_{i\in\{1,\ldots,L\}^j}
\prod_{k=1}^j
\big(
	z^{i_k} f^{(i_k)}
\big)
~=~
\sum_{\ell=0}^{\infty} z^{\ell}
\sum_{j=0}^{\ell} g^{(j)}
\sum_{\substack{i\in\{1,\ldots,L\}^j \\ \sum_{k=1}^j i_k = \ell}}
\prod_{k=1}^j
f^{(i_k)}
.
\end{align*}
Hence \eqref{eq:composition_expansion} holds with
\begin{align*}
G^{(\ell)}[f^{(1)},\ldots,f^{(\ell)}]
:=~&
\sum_{j=0}^{\ell} g^{(j)}
\sum_{\substack{i\in\{1,\ldots,L\}^j \\ \sum_{k=1}^j i_k = \ell}}
\prod_{k=1}^j
f^{(i_k)}
~=~
\sum_{j=0}^{\ell} g^{(j)}
\sum_{\substack{i\in\{1,\ldots,\ell\}^j \\ \sum_{k=1}^j i_k = \ell}}
\prod_{k=1}^j
f^{(i_k)}
\end{align*}
(note a change of index set for $i$) and
\begin{align*}
G^{(L^+)}[f^{(1)},\ldots,f^{(L)}; z]
:=~&
\sum_{\ell=L}^{\infty} z^{\ell-L}
\sum_{j=0}^{\infty} g^{(j)}
\sum_{\substack{i\in\{1,\ldots,L\}^j \\ \sum_{k=1}^j i_k = \ell}}
\prod_{k=1}^j
f^{(i_k)}.
\end{align*}
The latter expansion has the same majorant as $g(f(z))$, and therefore it is analytic on $[0,z_0]$

{\it Proof of \eqref{eq:composition_term_est}.}
Choose $r:=\min\{\frac13 \mu^{\alpha}, \frac13 R\}$.
Then for $|z| \geq r$, $z\in\bbC$, we have
\[
|f(z)|
\leq \sum_{\ell=1}^{L} c r^\ell
= c r \frac{1-r^{L}}{1-r}
\leq c r \frac{1}{1-\smfrac13}
= \smfrac{3}{2} c r \leq R/2.
\]
Hence using the Cauchy's integral we can represent
\[
G^{(\ell)} = \frac{1}{\ell!}\frac{\dd^\ell}{\dd z^\ell}g(f(z))\Big|_{z=0}
= \frac{1}{2\pi\i} \oint_{|z|=r} g(f(z)) z^{-1-\ell} \dd z
\]
and thus bound for $0<\mu \leq \mu_0$
\[
|G^{(\ell)}| \leq
r^{-\ell} \sup_{|\varphi|\leq R/2} |g(\varphi)|
\leq
\sup_{|\varphi|\leq R/2} |g(\varphi)| \, 3^\ell \, \max\{\mu^{-\alpha \ell},R^{-\alpha}\}
\leq C \mu^{-\alpha\ell}
.
\]

{\it Proof of \eqref{eq:composition_residual_est}.}
One can use a similar representation of the remainder:
\begin{align*}
\big|G^{(L^+)}(z)\big|
=~&
\frac{1}{2\pi\i} \oint_{|\zeta|=r} g(f(\zeta)) z^{-\ell} (z-\zeta)^{-1} \dd \zeta
\end{align*}
and hence bound, for $|z|<z_1 := z_0/2$,
$
\big|G^{(L^+)}(z)\big|
=
C r^{-\ell}
\leq
C \mu^{-\alpha L}.
$
The fact that $F^{(L^+)}$ is $\C^m$-regular for any $m$ follows from its boundedness and analyticity.
\end{proof}

\begin{proof}[Proof of Lemma \ref{lem:inverse_expansion}]
We take two steps. In the first step we note that the statement can be directly verified for $p=2$ as then $F(\theta)$ involves only a scalar product. In the second step we define
\[
\tilde{f}(z) := \|f(z^2)\|^2,
\qquad
g(\varphi) = \varphi^{p/2},
\]
(i.e., here we change variables $z=\sqrt{\theta}$) and apply Lemma \ref{lem:composition_expansion} to $\tilde{f}(z)$ and $g(\varphi)$ with $L+1$ instead of $L$, which proves the statements about $F^{(\ell)}$ and $F^{(1+L^+)}$. The statements on $\tilde{F}^{(\ell)}$ follow from a simple calculation.
\end{proof}

\section{Free energy calculation}\label{sec:Wfree}

In this section we show that \eqref{eq:Wfree} is the harmonic approximation of free energy per atom for an infinite chain, defined to be the limit as $N\to\infty$ for the $N$-periodic lattice.

In this section only, we consider $N$-periodic boundary conditions, instead of the free boundary conditions.
The space of such $N$-periodic lattice functions we denote by $\calU_\per(\calB_N)$, where we define $\calB_N = \{0,1,\ldots,N-1\}$ and $\calL_N = \calB_N + \frac12$.
Consider strains $u_n=\mF + v_n$, where $v\in\calU_\per(\calB_N)$.
Consider a harmonic potential energy in the form
\begin{align*}
E^\lin_{\theta,N}(v)
=~&
	N V(\mF,\mF) +
	\sum_{\xi\in\calL_N} \big(
		\smfrac12 V_{xx}(\mF,\mF) v_{\xi-\half}^2
		+
		V_{xy}(\mF,\mF) u_{\xi-\half} v_{\xi+\half}
		+
		\smfrac12 V_{yy}(\mF,\mF) v_{\xi+\half}^2
	\big)
\\ =~&
	N W(\mF) +
	\sum_{\xi\in\calL_N} \big(
		\smfrac12 \alpha_0 v_{\xi-\half}^2
		+
		\alpha_1 u_{\xi-\half} v_{\xi+\half}
		+
		\smfrac12 \alpha_0 v_{\xi+\half}^2
	\big)
,
\end{align*}
where $W(\mF) := V(\mF, \mF)$ is nothing but the standard Cauchy-Born energy density \cite{BlancLeBrisLions2002, EMing2007static, HudsonOrtner2012} and $\alpha_0 := V_{xx}(\mF,\mF) = V_{yy}(\mF,\mF)$ and $\alpha_1 := V_{xy}(\mF,\mF)$.
Our objective is to compute
\[
W^\free_{\theta,N} = -\frac{\theta}{N} \log\bigg( \int_{\calU_\per(\calB_N)} e^{-\theta^{-1} E(u)} \du \bigg),
\]
which is the free energy per atom (hence division by $N$) in the (formal) limit $N\to\infty$.

Denote by $w_k(x) = \frac1{\sqrt{N}}e^{2\pi i k x/N}$, $k=0,\ldots,N-1$, $x\in\calL_N$, the Fourier basis of $\calU_\per(\calB_N)$.
Then $u = \sum_{k=0}^{N-1} \tilde{u}_k w_k$, where $\tilde{u}_k$ are Fourier coefficients of $u$.
It is straightforward to establish that the orthonormal basis $w_k(x)$ diagonalizes the quadratic form.
We can express $D_r w_k = (e^{2\pi k r/N}-1)$ and hence
\[
E(u) = N V(\mF, \mF) + \sum_{k=0}^{N-1} \varkappa_k \tilde{u}_k^2
\]
where
$
\varkappa_k := \alpha_0 + \alpha_1 \cos(2\pi k/N).
$
Hence
\begin{align*}
W^\free_{\theta,N}
=~&
-\frac{\theta}{N} \log\bigg( \int_{\tilde{u}\in\bbR^N} e^{-\theta^{-1} W(\mF) -\theta^{-1} \sum_{k=0}^{N-1} \varkappa_k \tilde{u}_k^2} \dd\tilde{u} \bigg)
\\ =~&
-\frac{\theta}{N} \log\bigg(e^{-\theta^{-1} N W(\mF)} \int_{\tilde{u}\in\bbR^N} \prod_{k=0}^{N-1} e^{-\theta^{-1} \varkappa_k \tilde{u}_k^2} \dd\tilde{u} \bigg)
\\ =~&
W(\mF)-\frac{\theta}{N} \log\bigg(\prod_{k=0}^{N-1} \int_{\tilde{u}_k\in\bbR^N} e^{-\theta^{-1} \varkappa_k \tilde{u}_k^2} \dd\tilde{u}_k \bigg)
~=~
W(\mF)-\frac{\theta}{N} \log\bigg(\prod_{k=0}^{N-1} \Big(\frac{\pi\theta}{\varkappa_k}\Big)^{1/2} \bigg)
\\ =~&
W(\mF)+\frac{\theta}{2N} \sum_{k=0}^{N-1} \log \Big(\frac{\varkappa_k}{\pi\theta} \Big)
~=~
W(\mF)+\frac{\theta}{2} \log\Big(\frac{1}{\pi\theta}\Big) +\frac{\theta}{2N} \sum_{k=0}^{N-1} \log \varkappa_k
.
\end{align*}

In the limit $N\to\infty$, we have
\begin{align*}
W^\free_\theta(\mF)
=~& \lim_{N\to\infty} W^\free_{\theta,N}
~=~
W(\mF)+\frac{\theta}{2} \log\Big(\frac{1}{\pi\theta}\Big)
+\frac{\theta}{2} \int_0^1 \log (\alpha_0 + \alpha_1 \cos(2\pi \xi))
\dd\xi.
\end{align*}

The integral can be computed via the following lemma which is proved in a subsection below.
\begin{lemma}\label{lem:log_int}
For all $a\in\bbC$, $|a|\leq 1$,
\[
\int_0^1 \log\big(1+a^2-2 a \cos(2\pi \xi)\big) d\xi=0,
\]
where $\log(x)$ is defined to be analytic on $\bbC\setminus \{z\in\bbC : z\leq 0\}$
\end{lemma}

Indeed,
\begin{align*}
W^\free_\theta(\mF)
=~&
W(\mF)+\frac{\theta}{2} \log\Big(\frac{1}{\pi\theta}\Big)
+\frac{\theta}{2} \int_0^1 \log (\alpha_0 + \alpha_1 \cos(2\pi \xi))
\dd\xi,
\\ =~&
W(\mF)+\frac{\theta}{2} \log\Big(\frac{1}{\pi\theta}\Big)
-\frac{\theta}{2} \int_0^1 \log \smfrac{2 \left(\alpha_0-\sqrt{\alpha_0^2-\alpha_1^2}\right)}{\alpha_1^2} \, \dd\xi
\\~&
+\frac{\theta}{2} \int_0^1 \log \bigg( \smfrac{2 \left(\alpha_0-\sqrt{\alpha_0^2-\alpha_1^2}\right)}{\alpha_1^2} (\alpha_0 + \alpha_1 \cos(2\pi \xi)) \bigg)
\dd\xi,
\end{align*}
where the last integral can be shown to vanish thanks to Lemma \ref{lem:log_int} applied with $a=\smfrac{\alpha_0-\sqrt{\alpha_0^2-\alpha_1^2}}{\alpha_1}$.

Thus,
\begin{align*}
W^\free_\theta(\mF)
=~&
W(\mF)+\frac{\theta}{2} \log\Big(\frac{1}{\pi\theta}\Big)
-\frac{\theta}{2} \log \smfrac{2 \left(\alpha_0-\sqrt{\alpha_0^2-\alpha_1^2}\right)}{\alpha_1^2}
\\ =~&
W(\mF)+\frac{\theta}{2} \, \log\Big(\frac{1}{2\pi\theta}\Big)
-\frac{\theta}{2} \,
\log \smfrac{V_{xx}(\mF,\mF)-\sqrt{(V_{xx}(\mF,\mF))^2-(V_{xy}(\mF,\mF))^2}}{(V_{xy}(\mF,\mF))^2}
\end{align*}
which coincides with \eqref{eq:Wfree}.

\subsection{Proof of Lemma \ref{lem:log_int}}

\begin{proof}[Proof of Lemma \ref{lem:log_int}]
The stated identity is trivially true for $a=0$.
Denote
\[
F(a) := \int_0^1 \log\big(1+a^2-2 a \cos(2\pi \xi)\big) d\xi
\]
and compute, assuming $|a|<1$,
\[
\frac{\dd F}{\dd a}
=
\int_0^1 \frac{4 \pi a \sin(2\pi \xi)}{1+a^2-2 a \cos(2\pi \xi)} d\xi.
\]
(For $|a|=1$, the integral becomes improper and necessitates for extra care in differentiating.)

We can then turn $\frac{\dd F}{\dd a}$ into an integral over the unit circle $C:=\{z\in\bbC : |z|=1\}$ which can be computed using residues:
\begin{align*}
\frac{\dd F}{\dd a}
=~&
\int_0^1 \frac{-2 \pi i a (e^{2\pi i \xi}-e^{-2\pi i \xi})}{1+a^2-a (e^{2\pi i \xi}+e^{-2\pi i \xi})} d\xi
~=~
\oint_C \frac{-2 \pi i a (\zeta-\zeta^{-1})}{1+a^2-a (\zeta+\zeta^{-1})} \frac{1}{2\pi i \zeta}d\zeta
\\=~&
-2 \pi i a {\rm Res}\Big(
 \frac{ (\zeta-\zeta^{-1})}{1+a^2-a (\zeta+\zeta^{-1})} \frac{1}{2\pi i \zeta}
 , a\Big)
-2 \pi i a {\rm Res}\Big(
 \frac{ (\zeta-\zeta^{-1})}{1+a^2-a (\zeta+\zeta^{-1})} \frac{1}{2\pi i \zeta}
 , 0\Big)
\\=~&
2 \pi i - 2 \pi i = 0
.
\end{align*}
This proves that $F(a)=0$ for all $|a|<1$.

To extend this result to the case $|a|=1$, we need to apply the Lebesgue theorem to the sequence of functions $f_n(\xi) = \log(1+a_n^2-2 a_n \cos(2\pi \xi))$ with $a_n = \frac{n}{n+1} a$ ($n=1,2\ldots$).
Indeed, $f_n(\xi)$ converge to $\log(1+a^2-2 a \cos(2\pi \xi))$ for almost all $\xi$, $-\pi<{\rm Im}(f_n)<\pi$ from the definition of $\log$; ${\rm Re}(f_n)\leq \log(4)$ since $|1+a_n^2-2a_n \cos(2\pi \xi)|<4$; and
\begin{align*}
{\rm Re}(f_n)
\geq~&
\log(|1+a_n^2-2a_n \cos(2\pi \xi)|)
~=~
\log(|a_n|)+
\log(|a_n^{-1}+a_n-2 \cos(2\pi \xi)|)
\\ \geq~&
\log(1/2)+
\smfrac12 \log\big(
	{\rm Im}(a_n^{-1}+a_n)^2+
	{\rm Re}(a_n^{-1}+a_n-2 \cos(2\pi \xi))^2
\big)
\end{align*}
which can be shown, using the standard techniques of calculus, to be bounded below by a function with at most two logarithmic singularities.
Hence Lebesgue theorem applies and hence $F(a)=\lim_{n\to\infty} F(a_n)=0$.
\end{proof}

\bibliographystyle{plain}
\bibliography{atm,finite_temp}

\end{document}